\documentclass[10pt,a4paper]{amsart}

\pdfoutput=1
\usepackage{todonotes} 
\usepackage[
            pagebackref,
            pdfauthor   = {Giacomo Cherubini and Morten Risager},
            pdftitle    = {On the variance of the error term in the
              hyperbolic circle problem},
            pdfkeywords = {},
            pdfcreator  = {Pdflatex},
           ]
           {hyperref}
\hypersetup{colorlinks=true}
\usepackage{xcolor}
\usepackage{pdfpages}
\usepackage[normalem]{ulem}

\usepackage{amsfonts,amssymb,amsmath,amsthm}
\usepackage{url}
\usepackage{enumerate}

\numberwithin{equation}{section}

\usepackage{verbatim}
\usepackage{mathrsfs}

\newcommand\CC{\mathbb{C}}
\newcommand\HH{\mathbb{H}}

\newcommand\RR{\mathbb{R}}
\newcommand\ZZ{\mathbb{Z}}

\newcommand\eps{\varepsilon}

\theoremstyle{plain}
\newtheorem{teo}{Theorem}[section]
\newtheorem{lemma}[teo]{Lemma}
\newtheorem{cor}[teo]{Corollary}
\newtheorem{prop}[teo]{Proposition}

\theoremstyle{definition}
\newtheorem{definition}[teo]{Definition}

\newenvironment{rmk}[1][Remark]{%
\refstepcounter{teo}\begin{trivlist}\item[\hskip \labelsep {\scshape #1 \theteo.}]%
}{\end{trivlist}}

\title[Variance -- hyperbolic circle problem]{On the variance of the error term in the hyperbolic circle problem}

\author{Giacomo Cherubini}
\address{Department of Mathematical
  Sciences, University of Copenhagen, Universitetsparken 5, 2100
  Copenhagen \O, Denmark}
\email{giacomo.cherubini@math.ku.dk}

\author{Morten S. Risager}
\address{Department of Mathematical
  Sciences, University of Copenhagen, Universitetsparken 5, 2100
  Copenhagen \O, Denmark}
\email{risager@math.ku.dk}

\thanks{Both author were supported by a Sapere Aude grant from The Danish Council for Independent
Research (Grant-id:0602-02161B)}
\date{\today}

\makeatletter
\newcommand{\addresseshere}{%
  \enddoc@text\let\enddoc@text\relax
}
\makeatother

\makeatletter
\let\@wraptoccontribs\wraptoccontribs
\makeatother
\begin{document}

\begin{abstract}
Let  
$e(s)$ be the error term of the hyperbolic circle problem, and denote
by $e_\alpha(s)$ the fractional integral to order $\alpha$ of
$e(s)$.  We prove that for any small $\alpha>0$ the asymptotic variance of
$e_\alpha(s)$ is finite, and given by an explicit expression.   
Moreover, we prove that $e_\alpha(s)$ has a limiting distribution.
\end{abstract}

\maketitle


\section{Introduction}

Let $\HH$ be the hyperbolic plane, and denote by $d(z,w)$ be the
hyperbolic distance between $z,w\in \HH$.
For $\Gamma$ a cofinite Fuchsian group
 and $z,w\in\HH$, consider the function
\begin{equation}\label{def:N}
N(s,z,w):=\sharp\{\gamma\in\Gamma\;\mid\; d(z,\gamma w)\leq s\},
\end{equation}
which counts the number of translates $\gamma w$
 of $w$,  $\gamma\in\Gamma$
with hyperbolic distance from the point $z$ not exceeding $s$.
The hyperbolic lattice point problem asks for
the behaviour of $N(s,z,w)$ for big values of~$s$. It is known that
\[N(s,z,w)\sim \frac{\mathrm{vol}
(B_z(s))}{\mathrm{vol}(\Gamma\backslash\HH)}\]
as $s\to\infty$. Here $B_z(s)$ denotes the hyperbolic ball with center
$z$ and radius $s$. This can be proved in several ways, see
e.g. \cite[Section 1.3]{GorodnikNevo:2012}.

For our purposes it is convenient to appeal to the
spectral theory of the Laplace-Beltrami operator
\[\Delta = -y^2\left(\frac{\partial^2}{\partial x^2}+\frac{\partial^2}{\partial y^2}\right)\]
acting on a dense subset of $L^2(\Gamma\backslash\HH)$. The operator $\Delta$  has a discrete spectrum
\[0=\lambda_0<\lambda_1\leq\lambda_2\leq\lambda_3\leq\cdots\]
which is either finite or satisfies $\lambda_n\to\infty$, and a continuous spectrum which covers $[1/4,\infty)$
with multiplicity equal to the number of cusps of $\Gamma$.
The eigenvalues $\lambda_j\in (0,1/4)$ are called \emph{small
  eigenvalues}. Writing $\lambda_j=1/4+t_j^2$ with $\Im(t_j)\geq 0$
they correspond to $t_j$ in the complex segment
$t_j\in(0,i/2)$.
One defines the following main term: 
\begin{equation}\label{def:mainterm}
\begin{gathered}
M(s,z,w):=
\frac{\pi e^s}{\mathrm{vol}(\Gamma\backslash\HH)}
+
\sqrt{\pi} \sum_{t_j\in(0,\frac{i}{2})} \frac{\Gamma(|t_j|)}{\Gamma(3/2+|t_j|)} e^{s(1/2+|t_j|)} \phi_j(z)\overline{\phi_j(w})
\\
\phantom{x}+
4\big(s+2(\log 2-1)\big)\,e^{s/2}\,\sum_{t_j=0} \phi_j(z)\overline{\phi_j(w})
\\
\phantom{xxxxxxxxxxxxxxx}+
e^{s/2}\,\sum_{\mathfrak{a}} E_\mathfrak{a}(z,1/2)\overline{E_\mathfrak{a}(w,1/2})
\end{gathered}
\end{equation}
where $\phi_j$ is the eigenfunction associated to $\lambda_j$,
and $E_\mathfrak{a}(z,r)$ is the Eisenstein series associated to the cusp~$\mathfrak{a}$.
For the full modular group this expression simplifies to only the
first term, but for general groups
the small eigenvalues give rise to secondary terms in the expansion of the counting function $N(s,z,w)$.
It is an unpublished result of Selberg (for a proof see
e.g.\cite[Thm. 12.1]{iwaniec_spectral_2002}) that
\begin{equation}\label{Selberg's-bound}
N(s,z,w)-M(s,z,w)\ll e^{\frac{2s}{3}},
\end{equation}
and it is conjectured that the true size of the difference should be not bigger than $e^{s(\frac{1}{2}+\eps)}$
for any $\eps>0$. Define
\begin{equation}
  \label{normalized-reminder}
e_\Gamma(s,z,w)=\frac{N(s,z,w)-M(s,z,w)}{e^{s/2}}  
\end{equation}
to be the normalized remainder in the hyperbolic problem.

Phillips and Rudnick have shown (\cite[Theorem 1.1]{phillips_circle_1994}) that 
\[\lim_{T\to\infty}\frac{1}{T}\int_0^T e_\Gamma(s,z,z)ds = 0.\]
We refer to the quantity on the left as the first (asymptotic) moment
of $e_\Gamma(s,z,z)$.

\medskip

It is an open problem whether  the (asymptotic) \emph{variance}  of
$e_\Gamma(s,z,w)$ exists and, if so, if it is finite. More precisely
we are interested in knowing
if 
\[\mathrm{Var}(e_\Gamma)=\lim_{T\to\infty}\frac{1}{T}\int_T^{2T} |e_\Gamma(s,z,w)|^2\,ds\]
exists and is finite. Phillips and Rudnick remarked \cite[Section 3.8]{phillips_circle_1994}
that 
they cannot show that the variance is finite but they prove
non-zero \emph{lower} bounds.

In order to simplify notation, from now on we will write
$e(s)$ in place of $e_\Gamma(s,z,w)$, assuming that the group $\Gamma$
and the points $z,w\in\HH$ are fixed once and for all.

 A first result on the size of the variance is due to Chamizo
 see \cite[Corollary 2.1.1]{chamizo_applications_1996} who proves, using his large sieve in
 Riemann surfaces \cite{chamizo_large_1996}
\begin{equation}\label{intro:eq1}
\frac{1}{T}\int_T^{2T} |e(s)|^2 ds \ll T^2
\end{equation}
(one gets from his statement to \eqref{intro:eq1} by changing variable $X=2\cosh(s)$).
This doesn't prove finiteness of the variance of $e(s)$ but does shows that the integral in \eqref{intro:eq1} grows at most polynomially in $T$.
This is an improvement on what one gets by simply plugging Selberg's pointwise bound,
and it is consistent with the conjecture $e(s)\ll e^{\eps s}$.
We remark that~\eqref{intro:eq1} can be improved to a bound $\ll T$ by
using classical methods due to Cram\'er
\cite{cramer_mittelwertsatz_1922,cramer_uber_1922}. For details see \cite{Cherubini:2016}. 

\medskip
\begin{rmk}Cram\'er studied the analogous Euclidian problem
  \cite{cramer_uber_1922}, and in this case he was able to prove that
  the variance is finite and find an explicit expression for it. 
Like us, he also used a spectral expansion (coming in his case from
Poisson summation), but contrary to our case the ``eigenvalues'' are
explicitly known and the decay of the spectral coefficients is
favorable. One difficulty in proving finiteness of the variance in
our problem (using a spectral approach) relates 
to the following feature of
the problem: the spectral coefficients do not decay sufficiently fast
compared to the number of eigenvalues. The way we get around this
problem is to slightly improve the decay of the coefficients using
fractional integration. The formalism that we adopt follows the lines of \cite{samko_fractional_1993}.
\end{rmk}

\begin{definition}
Let $\varphi\in L^p([0,A])$ be a $p$-summable function on $[0,A]$ for $p\geq 1$,
and let $\alpha>0$ be a positive real number. The \emph{fractional integral} of order~$\alpha$
of $\varphi$ is defined for $x\in[0,A]$ as the function
\[
I_\alpha\varphi\,(x) = \frac{1}{\Gamma(\alpha)} \int_0^x \frac{\varphi(t)}{(x-t)^{1-\alpha}}dt.
\]
The function $I_\alpha\varphi(x)$ will also be denoted by $\varphi_\alpha(x)$.
\end{definition}

It is straightforward from the definition that the fractional integral of order $\alpha=1$
coincides with the regular integral. It is interesting to consider integrals of small order
$0<\alpha<1$ of a given function $\varphi$, because we have
\[
\begin{gathered}
\lim_{\alpha\to 0^+}\varphi_\alpha(x)=\varphi(x)\quad \text{for a.e. }x\in[0,A].\\ 
\lim_{\alpha\to 0^+} \|\varphi_\alpha -\varphi\|_p = 0.
\end{gathered}
\]

The first condition is easy to check by integration by parts when $\varphi$ is regular.
If we integrate the function $\varphi$ to a very small order, we expect thus the resulting
function $\varphi_\alpha$ to be close to the original function.
In addition to this, fractional integration enhances the properties of $\varphi$;
Indeed, if $0<\alpha<1$ and $\varphi\in L^p$ with $1<p<1/\alpha$,
then $\varphi_\alpha\in L^q$, for $q=p/(1-\alpha p)>p$,
and therefore $\varphi_\alpha$ has better
summability properties.
Moreover, if $\varphi\in L^\infty$, then $\varphi_\alpha$ is H\"older of exponent $\alpha$,
and in general, if $\varphi$ is H\"older of exponent $0\leq\rho\leq 1$, then
for $0<\alpha<1$ the function $\varphi_\alpha$ is
H\"older of exponent $\rho+\alpha$\footnote{ The case $\rho+\alpha=1$ is special, as in this situation $\varphi_\alpha$ is in a slightly bigger space than $H^1$ (see \cite[Ch. 1, \S 3.3, Cor 1]{samko_fractional_1993}).}.
Hence $\varphi_\alpha$ has better regularity properties than $\varphi$.
For a reference on these and other results about fractional integration, see \cite{samko_fractional_1993}.

\begin{definition}
Let $0<\alpha<1$. We define, for $s>0$, the $\alpha$-integrated normalized remainder
term in the hyperbolic lattice point counting problem as
\[e_\alpha(s,z,w):=I_\alpha e_\Gamma(s,z,w).\]
where the integration is with respect to the first $s$ variable.
\end{definition}
The function $e_\alpha(s,z,w)$ is well-defined since for every $A>0$
we have $e(s)\in L^1([0,A])$.
When the group $\Gamma$ and the points $z,w\in\HH$ are fixed,
we will simply write $e_\alpha(s)$.

\smallskip

We first prove a pointwise bound and an average result for
$e_\alpha(s)$ that are analogous to the results for $e(s)$:
\begin{teo}\label{intro:pointwise:theorem}
Let $\Gamma$ be a cofinite group, $z,w\in\HH$, and $0< \alpha<1$.
Then
$$
e_\alpha(s)\ll
\begin{cases}
e^{s(1-2\alpha)/(6-4\alpha)} & 0<\alpha<1/2,\\
s & \alpha =1/2,\\
1 & 1/2 <\alpha < 1.
\end{cases}
$$
The implied constant depends on $z,w$, and the group $\Gamma$.
\end{teo}

\begin{rmk}
When  $\alpha=0$ this is  Selberg's bound \eqref{Selberg's-bound}
(recall the normalization in \eqref{normalized-reminder}).
When  $\alpha>0$ the exponent gets smaller approaching 0 as
$\alpha$ increases to 1/2. For the threshold $\alpha=1/2$ a polynomial factor appears,
while for $\alpha>1/2$ the function $e_\alpha(s)$ becomes bounded.

\end{rmk}

\begin{teo}\label{intro:mean:theorem}
Let $\Gamma$ be a cofinite group, $z,w\in \HH$, and $0<\alpha<1$. Then
\[\lim_{T\to\infty}\frac{1}{T}\int_T^{2T} e_\alpha(s)ds = 0.\]
\end{teo}

\begin{rmk}
When $\alpha=0$ this corresponds to \cite[Theorem 1.1]{phillips_circle_1994}.The case $\alpha=1$ is delicate:
if the group is cofinite but not cocompact we cannot show that the limit stays bounded,
while if the group is cocompact then it is possible to show that the
limit exists and is finite.
\end{rmk}

To be able to prove finite variance for $e_\alpha(s)$ we need to make
assumptions on the Eisenstein series. More precisely we need to
assume, in the case where $\Gamma$ is cofinite but not cocompact, 
that for $v=z$ and $v=w$ 
we have
\begin{equation}
  \label{Eisenstein-assumption}
  \int_{1}^\infty \frac{
    |E_\mathfrak{a}(v,1/2+it)|^{2p}}{t^{(3/2+\alpha)p}}dt<\infty, \quad\textrm{
  for some $1<p<\min(2,\alpha^{-1})$, and all $\mathfrak{a}$.}
\end{equation}

\begin{teo}\label{intro:variance:theorem}
Let $0<\alpha<1$ and assume \eqref{Eisenstein-assumption}.
Then we have
\[
\lim_{T\to\infty} \frac{1}{T} \int_T^{2T} |e_\alpha(s)|^2 ds
=
2\pi \!\!\! \sum_{t_j>0\atop \scriptscriptstyle \textrm{distinct}} \frac{|\Gamma(it_j)|^2}{|t_j^\alpha\Gamma(3/2+it_j)|^2} \; \bigg|\sum_{t_{j'}=t_j}\phi_{j'}(z)\overline{\phi_{j'}(w)}\;\bigg|^2
\]
and the sum on the right is convergent.
\end{teo}

\begin{rmk}\label{intro:rmk-cofinite}
Condition \eqref{Eisenstein-assumption} holds true for congruence
groups: it is implied by the following stronger condition
$$|E_\mathfrak{a}(z,1/2+it)|\ll_z \lvert t\rvert^{1/2+\eps}, \qquad t\gg 1$$
which holds for congruence groups. 
For a proof see \cite[Lemma
2.1]{young_note_2015} or combine \cite[Eq. (2.4), ftnote
2.]{brumley-templier} with a Maass-Selberg type argument as in the
proof of \cite[Lemma 6.1]{petridis-risager}.

For cocompact groups \eqref{Eisenstein-assumption} is
vacuous, so Theorem \ref{intro:variance:theorem} holds unconditional in this case.

Condition \eqref{Eisenstein-assumption} holds also if $\alpha>1/2$ and
the Eisenstein series satisfy that they are bounded polynomially as $t\to\infty$. 
We note also, by using Theorem
\ref{intro:pointwise:theorem}, that  when $\alpha>1/2$ the asymptotic
variance is bounded.
\end{rmk}

\begin{rmk}
  It is a straightforward exercise to show that if $f\in
  L^1_{loc}([0,\infty))$ then $T^{-1}\int_{T}^{2T}f(s)ds\to A$ as
  $T\to \infty$ if and only if $T^{-1}\int_{0}^Tf(s)ds\to A$ as  $T\to
  \infty$. It follows that theorems
  \ref{intro:mean:theorem} and \ref{intro:variance:theorem} are true also if we replace the integral from $T$
  to $2T$ by the integral from $0$ to $T$. For various technical
  reasons it is convenient to consider the integral from $T$ to $2T$.
  
\end{rmk}

\begin{rmk}
If we take $\alpha=0$ we cannot prove that the infinite series appearing in Theorem
\ref{intro:variance:theorem} is convergent. However, for groups like
$\hbox{SL}_2(\mathbb Z)$ this follows from standard (but probably very hard) conjectures (see Section \ref{section:hybrid-limits}). For groups $\Gamma$ where
\begin{equation}\label{intro:V-def}
V=
2\pi \!\!\! \sum_{t_j>0\atop \scriptscriptstyle \textrm{distinct}}
\frac{|\Gamma(it_j)|^2}{|\Gamma(3/2+it_j)|^2} \;
\bigg|\sum_{t_{j'}=t_j}\phi_{j'}(z)\overline{\phi_{j'}(w)}\;\bigg|^2
< \infty,
\end{equation}
and where the Eisenstein contribution is ``small'' it is tempting to
speculate that V should be the variance of $e_\Gamma(s)$ i.e. that 
\begin{equation}\label{intro:alpha-limit}
\mathrm{Var}(e_\Gamma)=\lim_{\alpha\to 0^+}\mathrm{Var}(e_\alpha)=V.
\end{equation}
In fact, by comparison
of \eqref{intro:V-def} with the explicit expression of the variance of error terms in other problems
(see \cite{cramer_mittelwertsatz_1922,cramer_uber_1922,akbary_limiting_2014}),
the quantity $V$ seems the appropriate candidate for being the
variance of $e_\Gamma$.
\end{rmk}

Finally, we conclude with a distributional result on $e_\alpha(s)$
which we prove as a by-product of bounds which emerge in the proof
Theorem \ref{intro:variance:theorem}.
Given a function $g:\RR_{\geq 0}\to\RR$ we say that $g$ admits a \emph{limiting distribution}
if there exists a probability measure $\mu$ on $\RR$ such that
\[
\lim_{T\to\infty} \; \frac{1}{T} \;
\int_0^{T} f(g(s))ds
=
\int_\RR f \, d\mu
\]
holds for every bounded continuous function $f:\RR\to\RR$.

\begin{teo}\label{intro:limiting-distribution}
Let $0<\alpha<1$ 
and let $\Gamma$ be as in Theorem \ref{intro:variance:theorem}.
Then the function $e_\alpha(s)$ admits a limiting distribution $\mu_\alpha$.
For $\alpha>1/2$, $\mu_\alpha$ is compactly supported.
\end{teo}

\noindent
In view of Remark \ref{intro:rmk-cofinite},
the theorem applies to congruence groups and cocompact groups.

\begin{rmk}
The technique of regularizing functions that do not have sufficiently good properties is
standard in analytic number theory. This can often be done for instance by convolution
with some smooth functions $f_\eps$ that approximate a delta function as $\eps$ tends to zero,
and it is in particular this type of smoothing that is used in \cite{phillips_circle_1994}
when proving lower bounds on $e(s)$.
Using fractional integration corresponds to pushing the standard
method to its limit. 
Indeed, the pre-trace formula for the integrated function $e_\alpha(s)$ for 
$\alpha\leq1/2$ is not absolutely convergent, which is a characteristic of $e(s)$
but not of the smooth approximation.
The small improvements given by the $\alpha$ integration
allows to prove the above theorems.
\end{rmk}


\section{Preliminaries}\label{section:definitions}

We recall here some basic facts on automorphic functions.
Let $z,w\in\HH$, and consider the standard point-pair invariant
\[u(z,w)=\frac{|z-w|^2}{4\Im(z)\Im(w)}.\]

We have
\begin{equation}\label{preliminaries:d-to-u}
2u(z,w)+1=\cosh d(z,w)
\end{equation}
Let $\Gamma\leq\mathrm{PSL}(2,\RR)$ be a cofinite Fuchsian group.
For $k:\RR\to\RR$ rapidly decreasing the function
\[K(z,w)=\sum_{\gamma\in\Gamma}k(u(z,\gamma w))\]
is an automorphic kernel for the group $\Gamma$.
If we define $h(t)$ to be the Selberg--Harish-Chandra transform of $k(u)$,
which is defined as an integral transform of $k$ in three steps as follows
\[
q(v)=\int_v^{+\infty}\frac{k(u)}{(u-v)^{1/2}}du,
\quad
g(r)=2q\left(\sinh^2\frac{r}{2}\right),
\quad
h(t)=\int_{-\infty}^{+\infty} e^{irt}g(r)dr,
\]
then we have the following spectral expansion of $K(z,w)$,
usually referred to as the pre-trace formula (see \cite[Theorem 7.4]{iwaniec_spectral_2002}):
\begin{prop}
Let $(k,h)$ be a pair such that $h(t)$ is even, holomorphic on a strip 
$|\Im(t)|\leq 1/2+\eps$,
and with $h(t)\ll (1+|t|)^{-2-\eps}$ in the strip. We have the following expansion for $K(z,w)$:
\[
K(z,w)=\sum_{t_j} h(t_j)\phi_j(z)\overline{\phi_j(w)}
+
\frac{1}{4\pi}\sum_\mathfrak{a}\int_\RR h(t)E_\mathfrak{a}(z,1/2+it)\overline{E_\mathfrak{a}(w,1/2+it)}\,dt
\]
and the right hand side converges absolutely and uniformly on compact sets.
\end{prop}
\noindent
The absolute convergence is a consequence of the local Weyl's law
(See e.g. \cite[Lemma~2.3]{phillips_circle_1994})
\begin{equation}\label{local-Weyl-law}
\sum_{|t_j|<T}|\phi_j(z)|^2 + \frac{1}{4\pi}\sum_\mathfrak{a}\int_{-T}^T |E_\mathfrak{a}(z,1/2+it)|^2\,dt
\sim
c T^2
\end{equation}
as $T\to\infty$, for a positive constant $c>0$, together with the
assumption on the decay of $h$. 

\smallskip
Using the pre-trace formula it is possible to give upper bounds
on eigenfunctions averaged over short intervals. More precisely one
can show (see \cite[Eq. (13.8)]{iwaniec_spectral_2002})
\begin{equation}\label{variance:short-average-cusp-forms}
\sum_{T\leq t_j\leq T+1} |\phi_j(z)|^2\ll T.
\end{equation}
In our proofs we also need to consider
\begin{equation}
  \label{b_j-bound}
  b_j=\sum_{t_{j'}=t_j}\phi_{j'}(z)\overline{\phi_{j'}(w)}\ll t_j
\end{equation}
where the bound follows immediately from
\eqref{variance:short-average-cusp-forms}.
From \eqref{local-Weyl-law} we find immediately that
\begin{equation}\label{long-sum-over-b_j}
  \sideset{}{'}\sum_{T\leq t_j\leq 2T}\lvert b_j\rvert\ll T^2.
\end{equation}
Here and in the rest of the paper a prime on a sum indexed over $t_j$
means that in this sum the $t_j$ are  listed \emph{without} multiplicity,
i.e. $t_j\neq t_\ell$ for $j\neq\ell$. 

Consider the counting function
defined in \eqref{def:N}. This can be written as an automorphic kernel as
\[N(s,z,w)=\sum_{\gamma\in\Gamma}k_s(u(z,\gamma w))\]
where $k_s(u)=\bold{1}_{[0,(\cosh s-1)/2]}$ is the indicator function of the set
$[0,(\cosh s-1)/2]$. This agrees with \eqref{def:N} by virtue of \eqref{preliminaries:d-to-u}.
In particular
\[
k_s(u(z,w))=
\begin{cases}
1 & \text{if }d(z,w)\leq s\\
0 & \text{if }d(z,w)>s.
\end{cases}
\]
To study the error term $e(s,z,w)$ we want to use the pre-trace formula.
However the Selberg--Harish-Chandra transform of $k_s$ only decays as fast as $O((1+|t|)^{-3/2})$
(see \cite[Lemma 2.5 and 2.6]{phillips_circle_1994}) and the pre-trace formula
is therefore not absolutely convergent. The standard way to go around
this is by regularizing the function
$k_s$ sufficiently to ensure that the associated Selberg--Harish-Chandra transform
has better decay properties.

Since our purpose is to study the normalized remainder, we consider instead of $k_s(u)$ the function
$k_s(u)e^{-s/2}$. This gives rise to the function $N(s,z,w)e^{-s/2}$, 
and subtracting from it the normalized main term $M(s,z,w)e^{-s/2}$ we obtain $e_\Gamma(s,z,w)$.
The fractional integral of order $\alpha$ of $e_\Gamma(s,z,w)$ is defined
by $e_\alpha(s)=I_\alpha e_\Gamma(s,z,w)$.

\noindent
By linearity of the fractional integral we see that
\[
e_\alpha(s)=I_\alpha e_\Gamma(s,z,w)=
I_\alpha\left(\frac{N(s,z,w)}{e^{s/2}}\right)-I_\alpha\left(\frac{M(s,z,w)}{e^{s/2}}\right).
\]
It is easy to compute directly what the second term is. We have indeed for $\beta>0$
\begin{equation}\label{anotherbound}I_\alpha(e^{\beta s}) = \frac{e^{\beta s}}{\beta^\alpha}+O\left(\frac{1}{\beta\Gamma(\alpha)s^{1-\alpha}}\right)\end{equation}
and
\begin{equation}\label{I_alpha_s}
I_\alpha(s)=\frac{s^{\alpha+1}}{\Gamma(\alpha+2)},\quad I_\alpha(1)=\frac{s^\alpha}{\Gamma(\alpha+1)}.
\end{equation}
The implied constant in the first expression is absolute.
It is now natural to define the $\alpha$-integrated normalized main term to be
\[
\begin{gathered}
M_\alpha(s):=
\frac{\pi e^{s/2}}{2^{-\alpha}\mathrm{vol}(\Gamma\backslash\HH)}
+
\sqrt{\pi} \sum_{t_j\in(0,\frac{i}{2})} \frac{\Gamma(|t_j|)}{|t_j|^\alpha\Gamma(3/2+|t_j|)} e^{s|t_j|} \phi_j(z)\overline{\phi_j(w})
\\
\phantom{x}+
4\left(\frac{s^{\alpha+1}}{\Gamma(\alpha+2)}+\frac{2(\log 2-1)s^\alpha}{\Gamma(\alpha+1)}\right)\,\sum_{t_j=0} \phi_j(z)\overline{\phi_j(w})
\\
\phantom{xxxxx}+
\frac{s^\alpha}{\Gamma(\alpha+1)}
\sum_{\mathfrak{a}} E_\mathfrak{a}(z,1/2)\overline{E_\mathfrak{a}(w,1/2}).
\end{gathered}
\]

We then have
\[
M_\alpha(s)=I_\alpha \Big(\frac{M(s,z,w)}{e^{s/2}}\Big) + O\left(\frac{1}{\Gamma(\alpha) s^{1-\alpha}}\right)
\]
where the implied constant depends on $z,w$,
and the group $\Gamma$.
We conclude that the $\alpha$-integrated normalized remainder is expressed as
\[e_\alpha(s) = N_\alpha(s) - M_\alpha(s) +
O\left(\frac{1}{\Gamma(\alpha) s^{1-\alpha}}\right).\]
The function
$N_\alpha(s)=I_\alpha\big(N(s,z,w)e^{-s/2}\big)$
can be expressed
as an automorphic function associated to the kernel $k_\alpha(u)=I_\alpha(k_s(u)e^{-s/2})$
in the following way:
\begin{equation*}
N_\alpha(s) =
I_\alpha\Big(\frac{N(s,z,w)}{e^{s/2}}\Big) =
I_\alpha\left(\sum_{\gamma\in\Gamma} \frac{k_s(u(z,w))}{e^{s/2}}\right) =
\sum_{\gamma\in\Gamma} k_\alpha(u(z,w)).
\end{equation*}
Here we have used that the sum is finite, so we can interchange the order of integration and summation.
In order to apply the pre-trace formula we need to understand the properties
of the Selberg--Harish-Chandra transform $h_\alpha'(t)$
associated to $k_\alpha(u)$.
We have the following expression (see \cite[(1.62')]{iwaniec_spectral_2002})
for the Selberg-Harish-Chandra transform of a generic test function $k(u)$:
\[h(t)=4\pi\int_0^{+\infty} F_{1/2+it}(u) k(u) du,\]
where $F_\nu(u)$ is the hypergeometric function.
It follows, using that $k_s$ is compactly supported in $u$, that
\begin{align*}
h_\alpha'(t)
&= 4\pi \! \int_0^{+\infty} \! F_{1/2+it}(u) \, k_\alpha(u) \, du\\
&= 4\pi \! \int_0^{+\infty} \! F_{1/2+it}(u) \, I_\alpha (k_s(u)e^{-s/2}) \, du
= I_\alpha (h_s (t)e^{-s/2})
\end{align*}
where we have used that $I_\alpha k_s(u)$ is an integral in the variable $s$, and the double integral
is absolutely convergent so that we can interchange the order of integration.
The function $h_\alpha'(t)$ is the spectral function associated to the remainder $e_\alpha(s)$. 
We study in detail this function 
in section~\ref{section:analysis-of-shc}.


\section{Analysis of the Selberg-Harish-Chandra transform}\label{section:analysis-of-shc}

We prove some estimates on the function $I_\alpha(h_s(t)e^{-s/2})$ that will be useful to prove
pointwise and average results for the remainder function
$e_\alpha(s)$. 

\subsection{Integral representation}\label{integral-representation}

The Selberg-Harish-Chandra transform $h'_s(t)$ 
of the kernel $k'_s(u)=\bold{1}_{[0,(\cosh s-1)/2]}(u)e^{-s/2}$
is given for $s\geq 0$ by
(see \cite[(2.10)]{phillips_circle_1994} and \cite[(2.6)]{chamizo_applications_1996})
\[
h'_s(t) = \frac{2^{3/2}}{e^{s/2}} \int_{-s}^s (\cosh s - \cosh r)^{1/2} e^{irt} \, dr.
\]
It is an important but non-obvious feature of  $h'_s(t)$ that it decays as fast as $|t|^{-3/2}$ when $t\to\infty$. 
A flexible method to show such decay consists
in shifting the contour of integration from the interval $[-s,s]$ to a pair of vertical half-lines
in the complex plane, with base points $\pm s$. This is done in \cite[Lemma 2.5]{phillips_circle_1994},
and this method  
can be used  also to analyse  the fractional integral of $h'_s$.
Since the function $h'_s(t)$ is even, we will from now on only consider~$t>0$,
using the reflection formula $h'_s(t)=h'_s(-t)$ to get negative values of~$t$.
We have (see \cite[p. 89]{phillips_circle_1994})
\[
\begin{gathered}
h'_s(t) = 2\Re(J_s(t)),
\\
J_s(t) = -2i \int_0^\infty (1-e^{iv})^{1/2} (1-e^{-2s}e^{-iv})^{1/2} e^{-tv} dv e^{its}.
\end{gathered}
\]
It is convenient also to set $J_s(t)=0$ for $s<0$.
For technical reasons that will be clear later, it is convenient to consider a small shift
of the function $h'_s(t)$.
For $0\leq \delta<1$ and $s>2$ consider the function $h'_{s\pm\delta}(t)$.
We consider the fractional integral of order $\alpha$ of $h'_{s\pm\delta}(t)$,
which we will denote by $h'_{\alpha,s\pm\delta}(t)$, i.e
\begin{align}
\nonumber h'_{\alpha,s\pm\delta}(t)&=
\frac{1}{\Gamma(\alpha)}\int_{0}^s \frac{2\Re(J_{x\pm\delta}(t))}{(s-x)^{1-\alpha}}\,dx
\\
\label{shc:h_alpha}
&=
2\Re\left(\frac{-2i}{\Gamma(\alpha)} \int_0^\infty (1-e^{iv})^{1/2} e^{-tv}
\int_{s_0}^s \frac{(1-e^{-2(x\pm \delta)-iv})^{1/2} e^{it(x\pm\delta)}}{(s-x)^{1-\alpha}}  \, dx \, dv\right).
\end{align}
Here $s_0$ denotes the quantity $s_0=\max\{0,\mp\delta\}$.
We will consider the innermost integral and move the contour of integration
to two vertical half-lines in the upper half-plane of $\CC$ with base points 
$s_0,s$. In order to move the contour we define for $\eps>0$ the set
\[ \Omega_\eps = \{ z \in \CC : z=x+iy, \; s_0<x<s, \; y>0, \; |z-s|>\eps \}. \]
The integrand
\[f(z)=\frac{(1-e^{-2(z\pm\delta)-iv})^{1/2} e^{it(z\pm\delta)}}{(s-z)^{1-\alpha}}\]
is holomorphic on $\Omega_\eps$ and continuous on its boundary, so we can apply Cauchy's theorem
and get
\begin{equation}\label{shc:eps-contour-shifting}
\int_{s_0}^{s-\eps}f(z)dz = \int_{\ell_1} f(z)dz - \int_{\ell_{2,\eps}} f(z)dz - \int_{\gamma_\eps} f(z)dz
\end{equation}
where $\ell_1=\{z=s_0+iy,\;y\geq 0\}$, $\ell_{2,\eps}=\{z\in\partial\Omega_\eps,\; \Re(z)=s, \; |z-s|\geq\eps\}$
and $\gamma_\eps=\{z\in\partial\Omega_\eps:|z-s|=\eps\}$.
Since we can bound
\[\left| \int_{\gamma_\eps} f(z)dz \right| \leq \frac{\pi \eps^\alpha}{\sqrt{2}}\]
we see, taking the limit as $\eps\to 0$ in
\eqref{shc:eps-contour-shifting}, that the integral
over $[s_0,s]$ equals the integral over $\ell_1$ minus the integral over $\ell_2=\{s+iy,\;y\geq 0\}$.
This gives 
\begin{align*}
\Gamma&(\alpha) h'_{\alpha,s\pm\delta}(t)
\\
=&
2\Re\left(2\int_0^\infty (1-e^{iv})^{1/2} e^{-tv}
 \int_0^\infty \frac{(1-e^{-2(s_0\pm\delta+i\lambda/t)-iv})^{1/2}e^{-\lambda}}{((s-s_0)t-i\lambda)^{1-\alpha}} \frac{d\lambda dv}{t^\alpha} e^{it(s_0\pm\delta)}\right)
\\
&-
2\Re\left(2\int_0^\infty (1-e^{iv})^{1/2} e^{-tv}
 \int_0^\infty \frac{(1-e^{-2(s\pm \delta +i\lambda/t)-iv})^{1/2}e^{-\lambda}}{(-i\lambda)^{1-\alpha}} \frac{d\lambda dv}{t^\alpha}\,e^{it(s\pm\delta)}\right).
\end{align*}

In the rest of the section we will use this integral representation of $h'_{\alpha,s\pm\delta}(t)$ to
obtain pointwise bounds for $h'_{\alpha,s\pm\delta}(t)$, bounds for the average $\frac{1}{T}\int_T^{2T}h'_{\alpha,s\pm\delta}(t)ds$,
and for products $\frac{1}{T}\int_T^{2T} h'_{\alpha,s\pm\delta}(t_1)\overline{h'_{\alpha,s\pm\delta}(t_2)}ds$.
These will be used in section \ref{section:pointwise-estimates} to get pointwise estimates on $e_\alpha(s)$,
and in sections \ref{section:mean} and \ref{section:variance} to get estimates for the first and second moment of $e_\alpha(s)$.

In several of the proofs we will tacitly use the following elementary
inequalities to interpolate between different bounds:
\begin{align*}
  \min(a^{-1},b^{-1})&\leq 2/(a+b)\leq 2\min(a^{-1},b^{-1})\\
\min(c,d)&\leq c^\sigma d^{1-\sigma}
\end{align*}
valid for all $a,b,c,d>0$ and $0\leq \sigma \leq 1$.

\subsection{Pointwise bounds}

We now state and prove two lemmas in which we estimate pointwise the function $h'_{\alpha,s\pm\delta}(t)$.
One is uniform in $t$ but worse in $s$, while the second is sharper for $|t|\gg 1$
but it has a singularity when $t\to 0$.

\begin{lemma}\label{shc:pointwise:uniform}
Let $0<\alpha<1$ and let $t\in\RR$. For $0\leq\delta<1$ and $s>2$ we have
\[h'_{\alpha,s\pm\delta}(t) \ll s^{\alpha+1}\]
where the implied constant is absolute.
\end{lemma}
\begin{proof}
Since the function $h'_s(t)$ satisfies the bound $|h'_s(t)|\leq |h'_s(0)|\ll s$ for every $t\in\RR$
(see \cite[Lemma 2.2]{phillips_circle_1994})
and in view of the fact that fractional integration preserves inequalities,
we get for every $t\in\RR$
\[h'_{\alpha,s\pm\delta}(t)=I_\alpha\big(h'_{s\pm\delta}(t)\big)\ll I_\alpha(s) = \frac{s^{\alpha+1}}{\Gamma(\alpha+2)}\]
where the last equality follows from \eqref{I_alpha_s}.
\end{proof}

\begin{lemma}\label{shc:pointwise:t-nonzero}
Let $0<\alpha<1$ and $t\in\RR$, $t\neq 0$. Then for $0\leq\delta<1$ and $s>2$
\[h'_{\alpha,s\pm\delta}(t)=2\sqrt{\pi}\;\Re\left(\frac{\Gamma(it)e^{it(s\pm\delta)}}{(it)^\alpha\Gamma(3/2+it)}\right) + \ell(s,\delta,t) \]
where
\[\ell(s,\delta,t)=O\left(\frac{1}{|t|^{1+\alpha}(1+\sqrt{|t|})}\left( e^{-2s} + \frac{1}{(1+|st|^{1-\alpha}\Gamma(\alpha))}\right)\right)\]
and the implied constant is absolute.
\end{lemma}

\begin{proof}
We have from Section \ref{integral-representation} that  $h'_{\alpha,s\pm\delta}(t)=2\Re(L_1)+2\Re(L_2)+2\Re(L_3)$, where
\begin{align}
\nonumber L_1=&\frac{-2}{\Gamma(\alpha)}\int_0^\infty (1-e^{iv})^{1/2}e^{-tv}dv
\int_0^\infty \frac{e^{-\lambda}}{(-i\lambda)^{1-\alpha}}d\lambda\frac{e^{it(s\pm\delta)}}{t^\alpha},
\\
\label{shc:pointwise:integrals} L_2 = &\frac{2}{\Gamma(\alpha)} \, \int_0^\infty (1-e^{iv})^{1/2} e^{-tv}
\int_0^\infty \frac{(1-e^{-2(s_0\pm \delta+i\lambda/t)-iv})^{1/2}}{((s-s_0)t-i\lambda)^{1-\alpha}}e^{-\lambda} \, \frac{d\lambda  dv}{t^\alpha} e^{it(s_0\pm\delta)},
\\
\nonumber L_3 = &\frac{-2}{\Gamma(\alpha)} \, \int_0^\infty (1-e^{iv})^{1/2} e^{-tv}
\int_0^\infty \frac{[(1-e^{-2(s\pm \delta+i\lambda/t)-iv})^{1/2}-1]}{(-i\lambda)^{1-\alpha}}e^{-\lambda} \frac{d\lambda dv}{t^\alpha} e^{it(s\pm\delta)}.
\end{align}
Integrating $L_1$ in $v$ and $\lambda$, and using the
relation 
\[-2i\int_0^\infty(1-e^{iv})^{1/2}e^{-tv}dv = \sqrt{\pi}\frac{\Gamma(it)}{\Gamma(3/2+it)}\]
which can be proved using the functional equation of the Gamma function, its relation with the Beta function,
and a a change of path in the integration,
we obtain
\[L_1=\sqrt{\pi}\frac{\Gamma(it)}{\Gamma(3/2+it)}\frac{e^{it(s\pm\delta)}}{(it)^\alpha}\]
and we recover the main term in the lemma. The error $\ell(s,\delta,t)$ is then given by the sum of $2\Re(L_2)+2\Re(L_3)$.
Bounding by absolute value and using
$|(1-e^{iv})^{1/2}|\ll\min(1,v^{1/2})$
we get
\begin{align*}
\Re(L_2)&=O\left(\frac{1}{|t|^{1+\alpha}(1+\sqrt{|t|})}\frac{1}{(1+|st|^{1-\alpha}\Gamma(\alpha))}\right)
\\
\Re(L_3)&=O\left(\frac{e^{-2s}}{|t|^{1+\alpha}(1+\sqrt{|t|})}\right)
\end{align*}
with implied absolute constants, and the lemma is proven.
\end{proof}

\subsection{Average bounds}

We give now two lemmas to estimate the size of the average of $h'_{\alpha,s\pm\delta}(t)$.
As in the pointwise bounds, the first estimate is uniform in $t$ but worse in $s$.

\begin{lemma}
Let $0<\alpha<1$ and $T>2$, and let $t\in\RR$. Then for $0\leq\delta<1$
\[
\frac{1}{T}\int_T^{2T} h'_{\alpha,s\pm\delta}(t) ds
\ll
T^{\alpha+1}
\]
where the implied constant is absolute.
\end{lemma}
\begin{proof}
This follows directly by integrating the bound in Lemma~\ref{shc:pointwise:uniform}.
\end{proof}

\begin{lemma}\label{shc:lemma:average2}
Let $0<\alpha<1$, $T>2$, and $t\in\RR$, $t\neq 0$. For $0\leq\delta<1$ we have
\[
\frac{1}{T}\int_T^{2T} h'_{\alpha,s\pm\delta}(t) ds
\ll
\frac{1}{|t|^{1+\alpha}(1+\sqrt{|t|})}
\left(
\frac{1}{1+T|t|}+\frac{e^{-2T}}{T}+\frac{1}{1+|Tt|^{1-\alpha}\Gamma(\alpha)}
\right).
\]
with implied absolute constant.
\end{lemma}
\begin{proof}
This follows directly by integrating the expression in Lemma~\ref{shc:pointwise:t-nonzero}.
\end{proof}

\subsection{Products}

Finally we give three lemmas on the size of the average of products of the form
$h'_{\alpha,s\pm\delta}(t_1)\overline{h'_{\alpha,s\pm\delta}(t_2)}$.
The first is a uniform estimate in $t_1,t_2$, the second deals with
the diagonal $t_1=t_2$, and the third gives a bound for the off $t_1\neq t_2$.

\begin{lemma}\label{shc:products:lemma1}
Let $0<\alpha<1$, $T>2$, and let $t_1,t_2\in\RR$. Then for $0\leq\delta<1$
\[
\frac{1}{T}\int_T^{2T} h'_{\alpha,s\pm\delta}(t_1)\overline{h'_{\alpha,s\pm\delta}(t_2)} ds
\ll
T^{2+2\alpha}
\]
where the implied constant is absolute.
\end{lemma}
\begin{proof}
This follows from using the bound of Lemma~\ref{shc:pointwise:uniform} for both factors
and then integrating directly.
\end{proof}

\begin{lemma}\label{shc:products:lemma2}
Let $0<\alpha<1$, $T>2$, and $t\in\RR$, $t\neq 0$. Then for $0\leq\delta<1$
\begin{align*}
\frac{1}{T}\int_T^{2T} &\lvert h'_{\alpha,s\pm\delta}(t)\rvert^2 ds
=
4\pi\;\left|\frac{\Gamma(it)}{t^\alpha\Gamma(3/2+it)}\right|^2
\\
&+
O\left(
\frac{1}{|t|^{2+2\alpha}(1+|t|)}
   \left(
      \frac{1}{1+|Tt|} + \frac{e^{-2T}}{T} + \frac{1}{1+|Tt|^{1-\alpha}\Gamma(\alpha)}
   \right)
\right)
\end{align*}
where the implied constant is absolute.
\end{lemma}
 \begin{proof}
Recall from the proof of Lemma~\ref{shc:pointwise:t-nonzero} that we can write $h'_{\alpha,s\pm\delta}(t)=2\Re(L_1)+2\Re(L_2)+2\Re(L_3)$,
where $L_1,L_2,L_3$ are defined in~\eqref{shc:pointwise:integrals}.
In order to get an estimate on the integral of
$|h'_{\alpha,s\pm\delta}(t)|^2$ it suffices to analyse
the various products $L_iL_j$ and $L_i\overline{L_j}$ for $i,j=1,2,3$.
The product $L_1\overline{L_1}$ gives
\[\frac{1}{T}\int_T^{2T} L_1\overline{L_1} ds = 2\pi\;\left|\frac{\Gamma(it)}{t^\alpha\Gamma(3/2+it)}\right|^2\]
which gives the first term in the statement. In order to get the error term we need an estimate on all the other products.
We discuss one of the products
and we state the bounds that we get on the others.
Consider the product $L_1\overline{L_2}$. Then we have
\[
\begin{aligned}
\int_T^{2T} L_1\overline{L_2} ds
&=
\frac{2\sqrt{\pi}\,\Gamma(it)}{i^\alpha t^{2\alpha} \, \Gamma(3/2+it)\Gamma(\alpha)}
\int_0^\infty (1-e^{-iv})^{1/2}e^{-tv}
\\
&\times
\int_0^\infty (1-e^{-2s_0\mp2\delta+2i\lambda/t+iv})^{1/2} e^{-\lambda}
\int_T^{2T} \frac{e^{it(s-s_0)}}{((s-s_0)t+i\lambda)^{1-\alpha}} \;ds \;d\lambda\,dv
\end{aligned}
\]
Bounding everything in absolute value and using that $0<\alpha<1$ in order to bound
on one hand ${|(s-s_0)t+i\lambda|}^{\alpha-1}\leq \lambda^{\alpha-1}$ and on the other
$|(s-s_0)t+i\lambda|^{\alpha-1}\leq 2^{1-\alpha}|st|^{\alpha-1}$ we obtain
\[
\begin{aligned}
\frac{1}{T}\int_T^{2T} L_1\overline{L_2} ds
&\ll
\min\left\{
\frac{1}{|t|^{2+2\alpha}(1+|t|)},
\frac{1}{|t|^{2+2\alpha}(1+|t|)}\frac{1}{\Gamma(\alpha)|Tt|^{1-\alpha}}
\right\}
\\
&\ll
\frac{1}{|t|^{2+2\alpha}(1+|t|)}
\frac{1}{\big(1+\Gamma(\alpha)|Tt|^{1-\alpha}\big)}
\end{aligned}
\]
with implied absolute constant.
Now we list the estimates one can get for the other products $L_iL_j$ and $L_i\overline{L_j}$.
For shortening notation write
\[
A(i,j)=\frac{|t|^{2+2\alpha}(1+|t|)}{T}\int_T^{2T} L_i L_j ds,
\quad
B(i,j)=\frac{|t|^{2+2\alpha}(1+|t|)}{T}\int_T^{2T} L_i \overline{L_j} ds,
\]
so for instance $B(1,1)$ gives the main term and $B(1,2)$ is the case that we just discussed explicitly.
Then we have
\[
\begin{gathered}
A(1,1) \ll \frac{1}{1+|Tt|};
\quad
A(1,2), B(1,2) \ll \frac{1}{1+\Gamma(\alpha)|Tt|^{1-\alpha}};
\\
A(1,3), B(1,3) \ll \frac{e^{-2T}}{T};
\quad
A(2,2), B(2,2) \ll \frac{1}{1+\Gamma(\alpha)^2|Tt|^{2-2\alpha}};
\\
A(2,3), B(2,3) \ll \frac{e^{-2T}}{T} \frac{1}{\big(1+\Gamma(\alpha)|Tt|^{1-\alpha}\big)};
\quad
A(3,3), B(3,3) \ll \frac{e^{-4T}}{T}.
\end{gathered}
\]
All the implied constants are absolute. Summing up all the relevant bounds we conclude the proof of the lemma.
\end{proof}

\begin{lemma}\label{shc:products:lemma3}
Let $0<\alpha<1$, let $T>2$, and let $t_1,t_2\in\RR$, $t_1,t_2\neq 0$, $t_1\neq t_2$. Then
for $0\leq\delta<1$ and $s>2$ 
\[
\begin{gathered}
\frac{1}{T}\int_T^{2T} h'_{\alpha,s\pm\delta}(t_1)\overline{h'_{\alpha,s\pm\delta}(t_2)}ds
\;\ll\;
\frac{1}{|t_1t_2|^{1+\alpha}(1+\sqrt{|t_1|})(1+\sqrt{|t_2|})}
\phantom{xxxxxx}
\\
\phantom{xxxxxxxx}
\times
\left(
\frac{1}{1+T|t_1-t_2|}
+
\frac{1}{1+T|t_1+t_2|}
+
\frac{1}{1+\Gamma(\alpha)T^{2-2\alpha}|t_1t_2|^{1-\alpha}}
\right)
\end{gathered}
\]
where the implied constant is absolute.
\end{lemma}

\begin{proof}
We argue similarly to the proof of the previous lemma. Since we only care about upper bounds
it is convenient to consider the sum of the integrals $L_1+L_3$.
Let us call then $P$ the sum $P=L_1+L_3$. In analysing
the product $h'_{\alpha,s\pm\delta}(t_1)\overline{h'_{\alpha,s\pm\delta}(t_2)}$
we need to analyse the products $PP$, $P\overline{P}$, $L_2L_2$, $L_2\overline{L_2}$, and the mixed products
$PL_2$, $P\overline{L_2}$, where in writing the products we assume that one factor
is evaluted at  $t=t_1$ and the other at $t_2$.
We discuss the case $P\overline{L_2}$ and we list the bounds that we obtain for the other products.
We have
\begin{equation}\label{shc:products:t_1!=t_2}
\begin{gathered}
\int_T^{2T} P\,\overline{L_2}\, ds
=
\frac{-4}{(t_1t_2)^\alpha\Gamma(\alpha)^2}
\int_0^\infty (1-e^{iv})^{1/2}e^{-t_1v}
\int_0^\infty (1-e^{-iu})^{1/2}e^{-t_2u}
\\
\times
\int_0^\infty \frac{e^{-\lambda}}{(-i\lambda)^{1-\alpha}}
\int_0^\infty (1-e^{-2(s_0\pm \delta-i\mu/t_2)+iu})^{1/2} e^{-\mu}
\\
\times
\int_T^{2T} \frac{(1-e^{-2(s\pm\delta+i\lambda/t_1)-iv})^{1/2}e^{\pm i\delta(t_1-t_2)-it_2s_0}}{((s-s_0)t_2+i\mu)^{1-\alpha}} e^{ist_1}
ds \, d\mu\, d\lambda\, du\, dv.
\end{gathered}
\end{equation}
Bounding everything in absolute value, and using
$\lvert(1-e^{iv})^{1/2}\rvert\ll \min(1,v^{1/2})$ 
we get the estimate
\begin{equation}\label{shc:products:t_1!=t_2:IJ}
\frac{1}{T}\int_T^{2T} P\overline{L_2} ds
\ll
\frac{1}{|t_1|^{1+\alpha}(1+\sqrt{|t_1|})}
\frac{1}{|t_1|^{1+\alpha}(1+\sqrt{|t_1|})}
\min\left\{1,\frac{1}{\Gamma(\alpha)|Tt_2|^{1-\alpha}}\right\}.
\end{equation}
If we instead integrate by parts in the inner integral we get from the
exponential $e^{ist_1}$ extra decay in $t_1$. 
If we then
take absolute value we find the estimate
\[
\frac{1}{T}\int_T^{2T} P\overline{L_2} ds
\ll
\frac{1}{|t_1|^{1+\alpha}(1+\sqrt{|t_1|})}
\frac{1}{|t_1|^{1+\alpha}(1+\sqrt{|t_1|})}
\frac{1}{\Gamma(\alpha)|Tt_1| \, |Tt_2|^{1-\alpha}}.
\]
Interpolating this with the second bound in \eqref{shc:products:t_1!=t_2:IJ},
and combining the result with the first bound of \eqref{shc:products:t_1!=t_2:IJ},
we arrive at the symmetric bound in $t_1,t_2$
\[
\frac{1}{T}\int_T^{2T} P\overline{L_2} ds
\ll
\frac{1}{|t_1|^{1+\alpha}(1+\sqrt{|t_1|})}
\frac{1}{|t_1|^{1+\alpha}(1+\sqrt{|t_1|})}
\frac{1}{1+\Gamma(\alpha) T^{2-2\alpha} |t_1t_2|^{1-\alpha}}.
\]
The implied constant is absolute.
Similarly is proven that, denoting by $g(t)=|t|^{-1-\alpha}(1+\sqrt{|t|})^{-1}$, then
\[
\frac{1}{T}\int_T^{2T} PP ds \ll \frac{g(t_1)g(t_2)}{1+T|t_1+t_2|};\qquad
\frac{1}{T}\int_T^{2T} P\overline{P} ds \ll \frac{g(t_1)g(t_2)}{1+T|t_1-t_2|};
\]
\[
\Big|\frac{1}{T}\int_T^{2T} L_2L_2 ds\Big|+
\Big|\frac{1}{T}\int_T^{2T} L_2\overline{L_2} ds\Big|+
\Big|\frac{1}{T}\int_T^{2T} PL_2 ds\Big|
\ll
\frac{g(t_1)g(t_2)}{1+\Gamma(\alpha)T^{2-2\alpha}|t_1t_2|^{1-\alpha}}.
\]
All the implied constants are absolute. Summing up the relevant  estimates
finishes the proof.
\end{proof}


\section{Additional smoothing}\label{section:smoothing}

In order to have an absolutely convergent pretrace formula,
and be able thus to manipulate the spectral series termwise,
we need an automorphic kernel $K(z,w)=\sum k(u(z,w))$ such that
the Selberg--Harish-Chandra transform $h(t)$ of $k(u)$ is
decaying as fast as $|t|^{-2-\eps}$ as $t\to\infty$.
However we have seen in Lemma \ref{shc:pointwise:t-nonzero} that the function $h_\alpha'(t)$ only decays
as fast as $|t|^{-3/2-\alpha}$, and therefore for $\alpha\leq 1/2$ we don't get 
an absolutely convergent pretrace formula.
In this case we need to use additional smoothing in order to approximate
the remainder $e_\alpha(s)$. A standard procedure suffices for our purpose.

\subsection{Convolution smoothing}
Let $\delta>0$ and consider the function
\begin{equation}\label{def:k_delta}
\tilde{k}_\delta(u):=\frac{1}{4\pi\sinh^2(\delta/2)}\bold{1}_{[0,(\cosh(\delta)-1)/2]}(u)\nonumber
\end{equation}
where $\bold{1}_{[0,A]}$ is the indicator function of the set $[0,A]$.
It has unit mass, in the sense that
\[\int_\HH \tilde{k}_\delta(u(z,w))d\mu(z) = 1\nonumber.\]
Let $k_{s\pm\delta}(u)=\bold{1}_{[0,(\cosh(s\pm\delta)-1)/2]}(u)$
and define $k^\pm(u)$ as the functions given by
\[
k^\pm(u)
 := \left( k_{s\pm\delta}(u)\,\ast\,\tilde{k}_\delta\right)(u) \\
  = \int_\HH k_{s\pm\delta}(u(z,v)) \,\ast\,\tilde{k}_\delta(u(v,w)) \, d\mu(v).
\]
Using the triangle inequality $d(z,w)\leq d(z,v)+d(v,w)$, we observe
that when $Z>0$
 the convolution $k_Z(u)\,\ast\,\tilde{k}_\delta$ satisfies
\[
(k_Z\,\ast\,\tilde{k}_\delta) (u(z,w))
=
\begin{cases}
k_Z(u(z,w)) & d(z,w)\leq Z-\delta\\
0 & d(z,w)\geq Z+\delta.
\end{cases}\nonumber
\]
From this we deduce that for $z,w\in\HH$
\[
k^-(u(z,w)) \leq k_s(u(z,w)) \leq k^+(u(z,w))\nonumber
\]
and summing over $\gamma\in\Gamma$ we have
\begin{equation*}
N^-(s,\delta) := \sum_{\gamma\in\Gamma}k^-(u(z,\gamma w))
\; \leq \;
N(s,z,w)
\; \leq \;
\sum_{\gamma\in\Gamma}k^+(u(z,\gamma w)) =: N^+(s,\delta).
\end{equation*}
Defining $\tilde{e\,}^\pm(s):=\big(N^\pm(s,\delta)-M(s,z,w)\big)e^{-s/2}$ we obtain
\begin{equation}\label{useful-inequality}\tilde{e\,}^-(s) \; \leq \; e(s) \; \leq \; \tilde{e\,}^+(s),\end{equation}
where we recall that $e(s)=\left(N(s,z,w)-M(s,z,w)\right)e^{-s/2}.$
The advantage of taking a convolution smoothing is that the Selberg--Harish-Chandra transform
$h^\pm$ of the convolution kernel $k^\pm=k_{s\pm\delta}*\tilde{k}_\delta$ is the product $h^\pm(t)=h_{s\pm\delta}(t)\tilde{h}_\delta(t)$ of the
two Selberg--Harish-Chandra transforms $h_{s\pm\delta},\tilde{h}_\delta$ associated to
the kernels $k_{s\pm\delta},\tilde{k}_\delta$.
In \cite[Lemma 2.4]{chamizo_applications_1996}, an expression is given for $h_R(t)$ in terms of special functions.
We have, for every $R>0$ and every $t\in\CC$ such that $it\not\in\ZZ$,
\begin{equation}\label{h_R_special_functions}
h_R(t)=
2\sqrt{2\pi\sinh R}\;
\Re\left(
e^{its} \frac{\Gamma(it)}{\Gamma(3/2+it)} F\left(-\frac{1}{2};\frac{3}{2};1-it;\frac{1}{(1-e^{2R})}\right)
\right).
\end{equation}
For $t$ purely imaginary, $|t|<1/2$, we get (see \cite[Lemma 2.1]{phillips_circle_1994} and \cite[Lemma 2.4]{chamizo_applications_1996})
\begin{equation}\label{hR_t_imag}
h_R(t)=
\sqrt{2\pi\sinh R} e^{R|t|} \frac{\Gamma(|t|)}{\Gamma(3/2+|t|)}
+O\left(\big(1+|t|^{-1}\big) e^{R\left(\frac{1}{2}-|t|\right)}\right).
\end{equation}
For $R\leq 1$, there is a different expansion for $h_R(t)$ (see \cite[Lemma 2.4]{chamizo_applications_1996}).
Indeed, for $0\leq R\leq 1$ and $t\in\CC$ we can write
\begin{equation}\label{smallR}
h_R(t)=
2\pi R^2 \frac{J_1(Rt)}{Rt} \sqrt{\frac{\sinh R}{R}} + O\left(R^2 e^{R|\Im t|}\min\{R^2,|t|^{-2}\}\right).
\end{equation}
The expansion of $h_R(t)$ for small radius $R$ implies that the function $\tilde{h}_\delta(t)$ satisfies
\begin{equation}\label{htilde_estimate}
\tilde{h}_\delta(t)
=
\begin{cases}
1 + O(\delta|t|+\delta^2) & \delta|t|<1\\
O\left(\frac{1}{(\delta|t|)^{3/2}}\right) & \delta|t|\geq 1
\end{cases}
\end{equation}
when $\Im(t)$ is bounded.
Define
\begin{equation}\label{smoothing:M^pm}
\begin{aligned}
M^\pm(s,\delta) := \sum_{t_j\in[0,\frac{i}{2}]} & h^\pm(t_j) \phi_j(z)\overline{\phi_j(w)}
\\
&+
\frac{1}{4\pi} \sum_\mathfrak{a} E_\mathfrak{a}(z,1/2)\overline{E_\mathfrak{a}(w,1/2)}\int_\RR h^\pm(t)dt,
\end{aligned}
\end{equation}
and set
\begin{equation}\label{smoothing:e^pm}
e^\pm(s,\delta) := \frac{N^\pm(s,\delta)-M^\pm(s,\delta)}{e^{(s\pm\delta)/2}}.
\end{equation}
Using the estimates above we can now prove the following lemma:
\begin{lemma}\label{smoothing:lemma}
Let $s>0$ and $0<\delta<1$. Then there exists functions
$P^\pm(s,\delta)$ such that
\begin{equation*}
e^-(s,\delta) + P^-(s,\delta) \;\leq\; e(s) \;\leq\; e^+(s,\delta) + P^+(s,\delta)
\end{equation*}
Moreover there exist $0<\eps_\Gamma<1/4$ such that 
\[P^\pm(s,\delta) = O(\delta e^{s/2} + s\delta^{1/2} + e^{-\eps_\Gamma
  s})\]

The implied constants depend on $z,w$, and the group
$\Gamma$.
\end{lemma}

\begin{proof}\label{first-approximation}
Using \eqref{useful-inequality} we see that the inequality is
satisfied if we set
 \begin{align*}P^\pm(s,\delta)&=\frac{N^{\pm}(s,\delta)-M(s,z,w)}{e^{s/2}}-\frac{N^{\pm}(s,\delta)-M^\pm(s,\delta)}{e^{(s\pm\delta)/2}}
\\
&=\frac{M^\pm(s,\delta)-M(s,z,w)}{e^{s/2}}+\frac{\lvert N^{\pm}(s,\delta)-M^\pm(s,\delta)\rvert}{e^{s/2}}O(\delta).
\end{align*}
By discreteness there exist an $0<\eps_\Gamma<1/4$ such that any
imaginary $t_j\neq i/2$ satisfies $\eps_\Gamma \leq \lvert t_j\rvert
\leq 1/2-\eps_\Gamma$. Using the above expansions of $h_R(t)$ and $\tilde{h}_\delta$ together
with various Taylor expansions one can show
\begin{align*}
h^\pm(i/2)
&=
2\pi(\cosh s -1) \; + \; O(\delta e^s)
\\
h^\pm(0)
&=
4\big(s+2(\log 2-1)\big)e^{s/2}
\; + \;
O(s\,\delta\,e^{s/2}+e^{-s/2})\rule{0pt}{13pt}
\\
h^\pm(t_j)
&=
\sqrt{\pi} \frac{ \Gamma(|t_j|) }{\Gamma(3/2+|t_j|)} e^{s(1/2+|t_j|)} \rule{0pt}{13pt}
\; + \; 
O(\delta e^{s(1-\eps_\Gamma)}+e^{s(1/2-\eps_\Gamma)}),
\end{align*}
for $t_j\in(0,i/2)$, and via Fourier inversion we see that
\begin{equation*}
\int_\RR h^\pm(t)dt
=
4\pi e^{s/2} + O(\delta^{1/2} e^{s/2}+e^{-s/2}).
\end{equation*}
It follows that
\begin{equation}\label{smoothing:Mpm-vs-M}
M^\pm(s,\delta) = M(s,z,w) + O\left(1+\delta e^s + s\delta\, e^{s/2} + \delta^{1/2}e^{s/2} + e^{s(1/2-\eps_\Gamma)}\right).
\end{equation} From the pre-trace formula we find 
\begin{equation*}
  N^\pm(s,\delta)-M^{\pm}(s,\delta)=\sum_{t_j>0} h^\pm(t_j) \phi_j(z)\overline{\phi_j(w)}
+
\frac{1}{4\pi} \sum_\mathfrak{a} \int_\RR h^\pm(t)E_\mathfrak{a}(t)dt
\end{equation*}
where
\begin{equation}\label{regularization:E_frak}
E_\mathfrak{a}(t)
=
E_\mathfrak{a}(z,1/2+it)\overline{E_\mathfrak{a}(w,1/2+it)}
-
E_\mathfrak{a}(z,1/2)\overline{E_\mathfrak{a}(w,1/2)}
\end{equation}
We notice that by the previous expressions we find
\begin{equation}\label{hpm-bound}
 h^\pm(t) \ll \frac{e^{s/2}}{|t|^{1}(1+\sqrt{|t|})}\frac{1}{\big(1+|\delta t|^{3/2}\big)}
\end{equation}
 so we may indeed apply the pre-trace formula. Using \eqref{hpm-bound}
 and \eqref{local-Weyl-law} we find
 $N^\pm(s,\delta)-M^{\pm}(s,\delta)=O(e^{s/2}\delta^{-1/2})$. Combining
 these estimates the bound on $P^\pm(s,\delta)$ follows easily.
\end{proof}

Since we want to study the $\alpha$-integrated problem, we will
integrate the inequality in Lemma~\ref{smoothing:lemma} to get an analogous inequality in the $\alpha$-integrated case.
We set
\begin{equation}
e_\alpha^\pm(s,\delta)= I_\alpha (e^\pm(s,\delta)).\end{equation}
Integration 
now gives the following corollary:

\begin{cor}\label{smoothing:corollary}
Let $0<\alpha<1$ and let $0<\delta<1<s$. Then there exists functions
$P^\pm_\alpha(s,\delta)$ such that
\[
e_\alpha^-(s,\delta) + P_\alpha^-(s,\delta)
\;\leq\;
e_\alpha(s)
\;\leq\;
e_\alpha^+(s,\delta) + P_\alpha^+(s,\delta),
\]
where
\[P_\alpha^\pm(s,\delta) =  O\left( \delta e^{s/2} + s^{1+\alpha} \delta^{1/2} + s^\alpha e^{-\frac{\eps_\Gamma s}{2}} + \frac{1}{\Gamma(\alpha)s^{1-\alpha}}\right)\]
and the implied constant depends on $z,w$, and the group $\Gamma$.
\end{cor}

\begin{proof}
From Lemma \ref{first-approximation} and the fact that
integration preserves inequalities we see that the inequality is
satisfied for $P_\alpha^\pm(s,\delta)=I_\alpha(P^\pm(s,\delta))$.

Using now \eqref{anotherbound} and the bound
\begin{equation}
|I_\alpha(e^{-\beta s})|\leq \frac{(cs)^\alpha
  e^{-\beta(1-c)s}}{\Gamma(\alpha+1)}+\frac{1}{\Gamma(\alpha)(cs)^{1-\alpha}\beta}\quad\,\beta>0,\,\forall\;0<c<1
\end{equation}
we find, by integrating the
inequality in Lemma~\ref{smoothing:lemma}  and choosing $c=1/2$, the
desired bound on $P_\alpha^{\pm}(s,\delta)$. 
\end{proof}


\section{Pointwise Estimates}\label{section:pointwise-estimates}

In this section we prove Theorem
\ref{intro:pointwise:theorem}. 
We start by considering
the function $e_\alpha^\pm(s,\delta)$ constructed in Section~\ref{section:smoothing}.
From Corollary~\ref{smoothing:corollary} we conclude that
\begin{equation}\label{pointwise:inequality}
|e_\alpha(s)|\ll \max_\pm(|e_\alpha^\pm(s,\delta)|+|P_\alpha^\pm(s,\delta)|).
\end{equation}
We will prove an upper bound on the right-hand side, which will then
imply a bound on $e_\alpha(s)$.
Consider the function \begin{equation*} h_\alpha^{\prime\pm}(t) = I_\alpha\left(\frac{h_{s\pm\delta}(t)}{e^{(s\pm
  \delta)/2}}\right)\tilde{h}_\delta(t)=h'_{\alpha, s\pm \delta}(t)\tilde{h}_\delta(t)
\end{equation*}
where $h_{s\pm\delta}(t)$ and $\tilde{h}_\delta(t)$ are as in section \ref{section:smoothing}.
Using Lemma~\ref{shc:pointwise:t-nonzero} and~\eqref{htilde_estimate} we get the estimate
\[ h_\alpha^{\prime \pm}(t) \ll \frac{1}{|t|^{1+\alpha}(1+\sqrt{|t|})}\frac{1}{\big(1+|\delta t|^{3/2}\big)}. \]
We have therefore that $h_\alpha^{\prime \pm}(t)\ll |t|^{-2-\eps}$ decays fast enough to ensure absolute convergence
of the pretrace formula. Using the pre-trace formula and the
definition of $e_\alpha^\pm(s,\delta)$   we find
\begin{equation}\label{crucial-expansion}
e_\alpha^\pm(s,\delta)
=
\sum_{t_j>0} h_\alpha^{\prime \pm}(t_j) \phi_j(z)\overline{\phi_j(w)}
+
\frac{1}{4\pi} \sum_\mathfrak{a} \int_\RR h_\alpha^{\prime \pm}(t)E_\mathfrak{a}(t)dt
\end{equation}
(recall \eqref{regularization:E_frak}).

Consider first the discrete spectrum:
Using the decay of $h_\alpha^{\prime \pm}(t)$ we can split the sum at $t_j=\delta^{-1}$ and
using \eqref{local-Weyl-law}
and a standard dyadic decomposition we obtain the bound
\[
\begin{aligned}
\sum_{t_j>0} h_\alpha^{\prime \pm}(t_j) & \phi_j(z) \overline{\phi_j(w)}
\ll
\sum_{0<t_j<\frac{1}{\delta}} \frac{1}{t_j^{3/2+\alpha}}\big(|\phi_j(z)|^2+|\phi_j(w)|^2\big)
\\
&\phantom{xxxxxxxxxxxxxxx}+
\frac{1}{\delta^{3/2}}\sum_{t_j\geq\frac{1}{\delta}} \frac{1}{t_j^{3+\alpha}}\big(|\phi_j(z)|^2+|\phi_j(w)|^2\big)
\\
&\ll
\left(\frac{1}{\delta}\right)^{2-\frac{3}{2}-\alpha} \!\!\!\!\! + \;\; \frac{1}{\delta^{3/2}}\left(\frac{1}{\delta}\right)^{-1-\alpha}
\!\!\!\! + \; O(1)
\ll
\left(\frac{1}{\delta}\right)^{\frac{1}{2}-\alpha} \!\!\! + \; O(1).
\end{aligned}
\]
Similarly, using the analyticity of the Eisenstein series we have that
$E_\mathfrak{a}(t)=O(|t|)$
for $|t|<1$, and so using \eqref{local-Weyl-law} to bound the Eisenstein series
we obtain
\[
\begin{aligned}
\int_\RR \; h_\alpha^{\prime \pm}(t)\,E_\mathfrak{a}(t) dt
\ll
\int_{|t|<1} \frac{1}{|t|^\alpha}dt
&+
\int_{1\leq |t|<\frac{1}{\delta}} \frac{|E_\mathfrak{a}(t)|}{|t|^{3/2+\alpha}}dt
+
\frac{1}{\delta^{3/2}} \int_{|t|\geq \frac{1}{\delta}} \frac{|E_\mathfrak{a}(t)|}{|t|^{3+\alpha}} dt
\\
&\ll
\left(\frac{1}{\delta}\right)^{1/2-\alpha} + \; O(1).
\end{aligned}
\]
In the case when $\alpha=1/2$ a logarithmic term $\log\delta^{-1}$ instead of a power of $\delta$ appears.
Combining the result with \eqref{pointwise:inequality} we obtain
\[e_\alpha(s)=O\left(\delta e^{s/2} + s^{1+\alpha}\delta^{1/2} + \delta^{-1/2+\alpha} + 1\right).\]
The theorem follows by choosing $\delta=e^{-s/(3-2\alpha)}$.
For $\alpha=1/2$ we get $e_\alpha(s)=O(\delta e^{s/2}+s^{3/2}\delta^{1/2}+\log\delta^{-1}+1)$,
and $\delta=e^{-s/2}$ gives the result.

\begin{rmk}\label{pointwise:rmk2}
In the case $\alpha=1$ the analog result of Theorem \ref{intro:pointwise:theorem} differs on whether the group
$\Gamma$ is cocompact or cofinite but not cocompact. In the first case the proof works fine 
and we obtain that $e_1(s)=O(1)$. 
If $\Gamma$ is cofinite not cocompact, however, this type of proof doesn't provide $e_1(s)=O(1)$,
due to the contribution of the Eisenstein series near the point $t=0$.
We can indeed in this case only bound as follows:
\[
\begin{aligned}
\int_\RR h_1'^\pm(t) E_\mathfrak{a}(t) dt
&=
\int_{|t|<\eps} + \int_{\eps\geq |t|<1} + \int_{|t|\geq 1}
\\
&\ll
\int_{|t|<\eps} s^2 dt
+
\int_{\eps\leq |t|<1} \frac{1}{|t|} dt
+
\int_{|t|\geq 1} \frac{1}{|t|^{3/2}}dt
\\
&\ll
\eps s^2 + \log\frac{1}{\eps} + 1.
\end{aligned}
\]
Choosing $\eps=s^{-2}$ (and $\delta=e^{-s}$ to bound the error coming from the approximation of the main term)
we obtain that for $\Gamma$ a cofinite not cocompact group and $\alpha=1$ we have
\[e_1(s)\ll\log s\]
and hence we cannot show finiteness in this case.
\end{rmk}


\section{First moment of integrated normalized remainder}\label{section:mean}

In this section we prove Theorem \ref{intro:mean:theorem}.
We will show that for $\delta=e^{-T}$ we have
\begin{equation}\label{mean:theorem:eq1}
\lim_{T\to\infty}\frac{1}{T}\int_T^{2T}e_\alpha^\pm(s)ds
=
\lim_{T\to\infty}\frac{1}{T}\int_T^{2T}P_\alpha^\pm(s)ds
=0
\end{equation}
which will allow us to conclude, using Corollary
\ref{smoothing:corollary}, that
\[\lim_{T\to\infty}\frac{1}{T}\int_T^{2T}e_\alpha(s)ds=0.\]
The last part of \eqref{mean:theorem:eq1} is easily proven by direct integration
of the pointwise bounds on $P_\alpha^\pm(s)$ given in Lemma \ref{smoothing:lemma}. Indeed we have
\[
\begin{gathered}
\frac{1}{T}\int_T^{2T}P_\alpha^\pm(s)ds
=
O\left(\frac{1}{T}\int_T^{2T} \Big( \delta e^{s/2} + s^{1+\alpha} \delta^{1/2} + s^\alpha e^{-\frac{\eps_\Gamma s}{2}} + \frac{1}{\Gamma(\alpha)s^{1-\alpha}} \Big) ds\right)
\\
=
O\left(\frac{\delta e^T}{T}+T^{1+\alpha}\delta^{1/2}+T^\alpha e^{-\frac{\eps_\Gamma T}{2}}+\frac{1}{\Gamma(\alpha)T^{1-\alpha}}\right).
\end{gathered}
\]
Plugging $\delta=e^{-T}$ we get
\[
\frac{1}{T}
\int_T^{2T}P_\alpha^\pm(s)ds
\ll
\frac{1}{T} + \frac{T^{1+\alpha}}{e^{T/2}} + \frac{T^\alpha}{e^{\frac{\eps_\Gamma T}{2}}} + \frac{1}{\Gamma(\alpha)T^{1-\alpha}}
\]
which tends to zero as $T\to\infty$.

In order to analyze the first integral in \eqref{mean:theorem:eq1} use
again the expansion \eqref{crucial-expansion}.
Since the series and the integral are absolutely convergent, we can integrate termwise and obtain
\[
\int_T^{2T} \! e_\alpha^\pm(s,\delta)ds
=
\sum_{t_j>0} \int_T^{2T} \! h_\alpha^{\prime \pm}(t_j)ds \;\phi_j(z)\overline{\phi_j(w)}
+
\frac{1}{4\pi} \sum_\mathfrak{a} \int_\RR E_\mathfrak{a}(t) \int_T^{2T} \! h_\alpha^{\prime \pm}(t)ds \,dt,
\]
Using now Lemma \ref{shc:lemma:average2} and \eqref{htilde_estimate} we can bound
\begin{align*}
\frac{1}{T}\int_T^{2T} &h_\alpha^{\prime \pm}(t) ds
=
\frac{\tilde{h}_\delta(t)}{T}\int_T^{2T} h'_{\alpha,s\pm \delta}(t) ds
\\
&\ll
\frac{1}{|t|^{1+\alpha}(1+\sqrt{|t|})(1+|\delta t|^{3/2})}\left(\frac{1}{1+T|t|}+\frac{e^{-2T}}{T}+\frac{1}{1+|Tt|^{1-\alpha}\Gamma(\alpha)}\right).
\end{align*}
Consider the contribution of the discrete spectrum. We get
\begin{align*}\sum_{t_j>0} &\frac{1}{T}\int_T^{2T} h_\alpha^{\prime \pm}(t_j)ds\;\;\phi_j(z)\overline{\phi_j(w)}
\\
\ll&
\sum_{0<t_j\leq \delta^{-1}} \Big(\frac{1}{T|t_j|^{5/2+\alpha}}+\frac{e^{-2T}}{T|t_j|^{3/2+\alpha}}+\frac{1}{T^{1-\alpha}\Gamma(\alpha)|t_j|^{5/2}}\Big)\big(|\phi_j(z)|^2+|\phi_j(w)|^2\big)
\\
&+
\frac{1}{\delta^{3/2}}\sum_{t_j>\delta^{-1}}\Big(\frac{1}{T|t_j|^{4+\alpha}}+\frac{e^{-2T}}{T|t_j|^{3+\alpha}}+\frac{1}{T^{1-\alpha}\Gamma(\alpha)|t_j|^4}\Big)\big(|\phi_j(z)|^2+|\phi_j(w)|^2\big)
\\
&\ll
\frac{1}{T}+\frac{e^{-2T}}{T\delta^{1/2-\alpha}}+\frac{1}{T^{1-\alpha}\Gamma(\alpha)},
\end{align*}
where we have used \eqref{local-Weyl-law}. For $\delta=e^{-T}$ this
tends to zero as  $T\to\infty$.

Consider next the continuous spectrum.
Split the integral into three pieces, where we integrate respectively over $\{|t|\leq 1\}$,
$\{1<|t|\leq\delta^{-1}\}$, and $\{|t|>\delta^{-1}\}$. Let $\sigma=(1-\alpha)/2$ and $\lambda=1/2$.
We get
\begin{align*}
\frac{1}{4\pi}\sum_\mathfrak{a}\int_\RR E_\mathfrak{a}(t)&\frac{1}{T}\int_T^{2T}h_\alpha^{\prime \pm}(t)ds\,dt
\\
\ll
\int_{|t|\leq 1} &\Big(\frac{1}{T^\sigma |t|^{\alpha+\sigma}}+\frac{1}{\Gamma(\alpha)^\lambda T^{(1-\alpha)\lambda} |t|^{\alpha+(1-\alpha)\lambda}}+\frac{e^{-2T}}{T|t|^\alpha}\Big)dt
\\
&+
\int_{1<|t|\leq\delta^{-1}} \Big(\frac{1}{T |t|^{5/2+\alpha}}+\frac{e^{-2T}}{T|t|^{3/2+\alpha}}+\frac{1}{\Gamma(\alpha) T^{1-\alpha} |t|^{5/2}}\Big)|E_\mathfrak{a}(t)|\,dt
\\
&+
\frac{1}{\delta^{3/2}}\int_{|t|>\delta^{-1}} \Big(\frac{1}{T |t|^{4+\alpha}}+\frac{e^{-2T}}{T|t|^{3+\alpha}}+\frac{1}{\Gamma(\alpha) T^{1-\alpha} |t|^4}\Big)|E_\mathfrak{a}(t)|\,dt
\\
&\ll_\alpha
\frac{1}{T^{\sigma}} + \frac{e^{-2T}}{T\delta^{1/2-\alpha}}.
\end{align*}
Plugging $\delta=e^{-T}$ and taking the limit as $T\to\infty$ we get zero.
Putting together the discrete and continuous contributions concludes
the proof of Theorem \ref{intro:mean:theorem}.


\section{Computing the variance}\label{section:variance}

We have proved a pointwise bound and a mean value result for $e_\alpha(s)$.
Now we look at the second moment of $e_\alpha(s)$.
We start by using the pre-trace formula to write $e_\alpha^\pm(s)$  in
the following way:
\begin{equation}\label{variance:f_alpha+g_alpha}
e_\alpha^\pm(s)=f_\alpha(s,\delta)+g_\alpha^\pm(s,\delta)+Q_\alpha^\pm(s,\delta)
\end{equation}
where 
\begin{equation}
f_\alpha(s,\delta) := \sum_{0<t_j<\delta^{-1}} \Re\big(r_\alpha(t_j)e^{it_js}\big) \phi_j(z)\overline{\phi_j(w)},
\end{equation}
 with
\begin{equation}\label{variance:r_j-definition}
r_\alpha(t)=\frac{2\sqrt{\pi}\,\Gamma(it)}{(it)^\alpha\,\Gamma(3/2+it)},
\end{equation}
and the functions $g_\alpha^\pm(s)$ are defined by 
\begin{align*}
g_\alpha^\pm(s,\delta)&=A^\pm(s,\delta) + B^\pm(s,\delta), \textrm{ where}
\\
A^\pm(s,\delta)&=\sideset{}{'}\sum_{t_j\geq \delta^{-1}} \tilde{h}_\delta(t_j)h'_{\alpha,s\pm\delta}(t_j) b_j
\\
B^\pm(s,\delta)&=\sideset{}{'}\sum_{0<t_j< \delta^{-1}} \Big(\tilde{h}_\delta(t_j)h'_{\alpha,s\pm\delta}(t_j)-\Re(r_\alpha(t_j)e^{it_js})\Big) b_j.
\end{align*}
Here $b_j$ is defined in \eqref{b_j-bound}, $h'_{\alpha,s\pm\delta}(t)$ in \eqref{shc:h_alpha},
and $\tilde{h}_\delta(t)$ is as in section \ref{section:smoothing}. The
functions $Q_\alpha^\pm(s,\delta)$ are  the contributions coming from the Eisenstein series, given by
\[Q_\alpha^\pm(s,\delta)=\sum_\mathfrak{a}\frac{1}{4\pi}\int_\RR
h_\alpha^{\prime \pm}(t)E_\mathfrak{a}(t)dt,\]
and $E_\mathfrak{a}(t)$ is as in \eqref{regularization:E_frak}.

In bounding the integral of the square of these terms we will often
need the following simple estimate, which is extrapolated from
\cite{cramer_mittelwertsatz_1922, cramer_uber_1922}:
\begin{lemma}\label{cramer-lemma}
  For $a>1$ and a given $t_j>0$ we have
  \begin{equation*}
    \sideset{}{'}\sum_{t_j<t_\ell}\frac{\lvert b_\ell\rvert }{t_\ell^a(1+T{\lvert
        t_\ell-t_j\rvert})}\ll \frac{1}{t_j^{a-1}}\left(1+\frac{1}{T(a-1)}+\frac{\log(t_j+1)}{T}\right).
  \end{equation*}
  For $0\leq c\leq 1$ and a given $t_j>0$ we have
  \begin{equation*}
    \sideset{}{'}\sum_{t_j<t_\ell\leq R}\frac{\lvert b_\ell\rvert }{t_\ell^c(1+T{\lvert
        t_\ell-t_j\rvert})}\ll
    t_j^{1-c}\left(1+\frac{\log(R+1)}{T}\right)+\frac{R^{1-c}}{T(1-c)}\quad 0\leq c<1,
\end{equation*}
and the last term is to be replaced by $T^{-1}\log(R+1)$ if $c=1$.
The implied constants are absolute.
\end{lemma}
\begin{proof}Using \eqref{b_j-bound} we find
  \begin{align*}
    \sideset{}{'}\sum_{t_j<t_\ell}\frac{|b_\ell|}{t_\ell^a(1+T\lvert t_\ell-t_j\rvert)}
    &=
    \sideset{}{'}\sum_{t_j<t_\ell\leq t_j+1} + \;\sum_{n=1}^\infty\;\;\sideset{}{'}\sum_{n<t_\ell-t_j\leq n+1}
    \!\!\ll
    \frac{1}{t_j^{a-1}} + \frac{1}{T}\sum_{n=1}^\infty\frac{t_j+n}{(t_j+n)^an}
    \\
    &\ll
    \frac{1}{t_j^{a-1}} + \frac{1}{T}\sum_{n\leq t_j}\frac{1}{(t_j+n)^{a-1}n}
    +
    \frac{1}{T}\sum_{n\geq t_j}\frac{1}{(t_j+n)^{a-1}n}
    \\
    &\ll
    \frac{1}{t_j^{a-1}}+\frac{1}{Tt_j^{a-1}}\sum_{n\leq t_j}\frac{1}{n}+\frac{1}{T}\sum_{n\geq t_j}\frac{1}{n^a}
    \\
    &\ll
    \frac{1}{t_j^{a-1}}\left(1+\frac{1}{T(a-1)}+\frac{\log(t_j+1)}{T}\right).
  \end{align*}
The second statement is proved analogously.
\end{proof}
We remark that by the above lemma, a symmetry argument, and partial
summation and \eqref{local-Weyl-law} we find
from the above lemma that for $a>3/2$  and $c>0$ large we have
\begin{equation}
  \label{cramer-goodie}
\sideset{}{'}\sum_{\substack{t_j, t_\ell\geq c\\t_j\neq
    t_\ell}}\frac{\lvert b_jb_\ell\rvert }{(t_jt_\ell)^a}\frac{1}{(1+T\rvert t_j-t_\ell\lvert)}\ll
\frac{c^{3-2a}}{2a-3}\left(1+\frac{\log c}{T}\right)
\end{equation}
The implied constant depends on $z$,$w$, and $\Gamma$. 

We are now ready to show that the functions $g_\alpha^\pm(s,\delta)$ are small on average. More
precisely we have the following lemma:

\begin{lemma}\label{variance:lemma}
For $\delta=e^{-T}$ we have
\[\lim_{T\to\infty}\frac{1}{T}\int_T^{2T} |g_\alpha^\pm(s,\delta)|^2ds=0.\]
\end{lemma}

\begin{proof}

Using \eqref{htilde_estimate}, Lemmata ~\ref{shc:products:lemma2} and
~\ref{shc:products:lemma3} together with \eqref{b_j-bound},
\eqref{local-Weyl-law} and \eqref{cramer-goodie}
we find
\begin{align}\label{variance:lemma-eq1}
\frac{1}{T}\int_T^{2T}& \!\! |A^\pm(s,\delta)|^2ds
=\!\!\!\!
\sideset{}{'}\sum_{t_j,t_\ell\geq \delta^{-1}}
\tilde{h}_\delta(t_j) \tilde{h}_\delta(t_\ell) \, b_j \overline{b_\ell} \,
\frac{1}{T}\int_T^{2T} h'_{\alpha,s\pm\delta}(t_j)\overline{h'_{\alpha,s\pm\delta}(t_\ell)} ds
\nonumber\\
&\ll
\sideset{}{'}\sum_{t_j\geq\delta^{-1}}
\frac{|b_j|^2}{1+|\delta t_j|^3}
\frac{1}{|t_j|^{3+2\alpha}}
\nonumber\\
&\qquad+
\underset{t_j\neq t_\ell}{\sideset{}{'}\sum_{t_j\geq\delta^{-1}}\sideset{}{'}\sum_{t_\ell\geq\delta^{-1}}}
\frac{|b_jb_\ell|}{|t_j t_\ell|^{3/2+\alpha}(1+|\delta t_j|^{3/2})(1+|\delta t_\ell|^{3/2})}
\\
&\quad\quad\quad\times
\left(\frac{1}{1+T|t_j-t_\ell|}+\frac{1}{1+\Gamma(\alpha)T^{2-2\alpha}|t_jt_\ell|^{1-\alpha}}\right)
\nonumber\\
&\ll
\delta^{2\alpha}
+
  \delta^{2\alpha}\left(1+\frac{\log{\delta^{-1}}}{T}\right)+ \frac{\delta}{\Gamma(\alpha)\, T^{2-2\alpha}} . 
\nonumber
\end{align}
The implied constant doesn't depend on $\alpha$.
Choosing $\delta=e^{-T}$ and taking the limit as $T\to\infty$ we get zero.

For the analysis of $B^\pm(s,\delta)$ a long and tedious
computation like in the proof of Lemma~\ref{shc:products:lemma2} and
Lemma~\ref{shc:products:lemma3} shows that
for $0< \delta< 1$ we have
\begin{align*}
 \frac{1}{T}&\int_T^{2T}
 \Big(\tilde{h}_\delta(t_j)h'_{\alpha,s\pm\delta}(t_j)-\Re(r_\alpha(t_j)e^{it_js})\Big)
 \overline{\Big(\tilde{h}_\delta(t_\ell)h'_{\alpha,s\pm\delta}(t_\ell)-\Re(r_\alpha(t_\ell)e^{it_\ell s})\Big)}\,ds \\
&\ll
 \frac{1}{|t_jt_\ell|^{3/2+\alpha}}\left(\frac{(\delta
  |t_j|+\delta^2)(\delta |t_\ell|+\delta^2)+e^{-2T}(\delta
  |t_\ell|+\delta^2)+(\delta
  |t_j|+\delta^2)e^{-2T}+e^{-4T}}{1+T|t_j-t_\ell|} \right.\\
&\left.+ \frac{1}{1+\Gamma(\alpha)T^{2-2\alpha}|t_jt_\ell|^{1-\alpha}}\right)
\end{align*}
where we have used the estimate 
\begin{equation*}|e^{\pm i\delta t_j}-\tilde{h}_\delta(t_j)| = O\big(\delta |t_j| + \delta^2\big).\end{equation*}
With this we can estimate, with the same reasoning used in bounding $A^\pm(s,\delta)$,
and choosing $\delta=e^{-T}$,
\begin{equation}\label{variance:lemma-eq2}
\frac{1}{T}\int_T^{2T} |B^\pm(s,\delta)|^2\,ds
\ll
\begin{cases}
\frac{\delta^{2\alpha}}{2-2\alpha}\left(1+\frac{1}{T|\alpha-1/2|}\right) + \frac{1}{\Gamma(\alpha)T^{2-2\alpha}} & \alpha\neq 1/2\\
\delta + \frac{1}{\Gamma(\alpha)T^{2-2\alpha}} & \alpha=1/2.
\end{cases}
\end{equation}
The implied constant doesn't depend on $\alpha$.
As $\delta=e^{-T}$, taking the limit as $T\to\infty$ this goes to zero, and
this proves the lemma. 
\end{proof}

\subsection{Variance, cocompact groups}\label{subsection-variance-cocompact}
We are now ready to prove -- in the co-compact case -- that the
variance of $e_\alpha(s)$ is finite.
By Corollary~\ref{smoothing:corollary} we find 
\begin{equation*}
\lvert e_\alpha(s)-f_\alpha(s,\delta)\rvert
\leq
\max_\pm\big\{\,|g_\alpha^\pm(s,\delta)+P^\pm(s,\delta)|\,\big\}.
\end{equation*}
Now we claim that for $\delta=\delta(T)=e^{-T}$ we have
\[
\lim_{T\to\infty}\frac{1}{T}\int_T^{2T} |g_\alpha^\pm(s,\delta)|^2ds
=
\lim_{T\to\infty}\frac{1}{T}\int_T^{2T} |P_\alpha^\pm(s,\delta)|^2ds
=0.
\]
The first limit is proven in Lemma \ref{variance:lemma}, while the
second limit can be proven by using the pointwise bound on
$P_\alpha^\pm(s,\delta)$ from Corollary \ref{smoothing:corollary}, since
\begin{equation}\label{variance:cocompact-P-square-bounds}
\begin{aligned}
\frac{1}{T}\int_T^{2T} |P_\alpha^\pm(s,\delta)|^2ds
&\ll
\frac{1}{T}\int_T^{2T} \Big(\delta e^{s/2} + s^{1+\alpha} \delta^{1/2} + s^\alpha e^{-\frac{\eps_\Gamma s}{2}} + \frac{1}{\Gamma(\alpha)s^{1-\alpha}}\Big)^2ds
\\
&\ll
\frac{\delta^2 e^{2T}}{T} + T^{2+2\alpha}\delta + \frac{T^{2\alpha}}{e^{\eps_\Gamma T}} + \frac{1}{\Gamma(\alpha)^2 T^{2-2\alpha}},
\end{aligned}
\end{equation}
so that choosing $\delta=e^{-T}$ and taking the limit as $T\to\infty$ we get zero.
The implied constant is independent of $\alpha$.
If we can now compute the second moment of $f_\alpha(s,\delta)$ for $\delta=e^{-T}$
and show that it is asymptotically finite, then we may conclude
\[
\lim_{T\to\infty} \frac{1}{T} \int_T^{2T} |e_\alpha(s)|^2 ds
=
\lim_{T\to\infty} \frac{1}{T} \int_T^{2T} |f_\alpha(s,\delta)|^2 ds.
\]
The explicit expression for the right-hand side will give the sum appearing
in the statement of the theorem, and this will conclude the proof.
The problem therefore reduces to computing the second moment of $f_\alpha(s,\delta)$ for $\delta=e^{-T}$.
For this we follow Cram\'er~\cite[p.149-150]{cramer_mittelwertsatz_1922}
and Landau~\cite[Proof of Satz 476]{landau_vorlesungen_1969}. 
We can  write
\begin{equation}\label{variance:f_alpha-definition}
f_\alpha(s,e^T) = \sideset{}{'}\sum_{0<t_j<e^T} \Re(r_\alpha(t_j)e^{it_js})b_j,
\end{equation}
 We obtain
\begin{equation}\label{variance:double-sum}
\begin{gathered}
\frac{1}{T} \int_T^{2T} |f_\alpha(s,\delta)|^2 ds
=
\sideset{}{'}\sum_{0<t_j<e^T} \sideset{}{'}\sum_{0<t_\ell<e^T} b_j\overline{b_\ell} \frac{1}{T}\int_T^{2T} \Re(r_\alpha(t_j)e^{it_js})\Re(r_\alpha(t_\ell) e^{it_\ell s}) \, ds
\\
=\frac{1}{2}\sideset{}{'}\sum_{0<t_j<e^T} |b_jr_\alpha(t_j)|^2
+O\left(\sideset{}{'}\sum_{0<t_j<e^T}\frac{|b_jr_\alpha(t_j)|^2}{T|t_j|}\right)
+O\left(\underset{t_j\neq t_\ell}{\sideset{}{'}\sum_{0<t_j<e^T}\sideset{}{'}\sum_{0<t_\ell<e^T}} \frac{|b_jb_\ell r_\alpha(t_j)r_\alpha(t_\ell)|}{1+T|t_j-t_\ell|}\right).
\end{gathered}
\end{equation}
The middle sum is bounded (uniformly in $\alpha$)   
by $O(T^{-1})$,
while for $T>1$ the last sum is clearly bounded by
\[
\underset{t_j\neq t_\ell}{\sideset{}{'}\sum_{0<t_j<e^T}\sideset{}{'}\sum_{0<t_\ell<e^T}} \;
\frac{|b_jb_\ell r_\alpha(t_j)r_\alpha(t_\ell)|}{1+T|t_j-t_\ell|}
=
O\left(
\underset{t_j\neq t_\ell}{\sideset{}{'}\sum_{0<t_j}\sideset{}{'}\sum_{0<t_\ell}} \;
\frac{|b_jb_\ell r_\alpha(t_j)r_\alpha(t_\ell)|}{1+|t_j-t_\ell|}
\right).
\]

If the last sum is finite we may use the dominated convergence to conclude, since each term goes to
zero as $T\to\infty$, that the left-hand side is $o(1)$.  To prove finiteness, 
\eqref{cramer-goodie} and \eqref{variance:short-average-cusp-forms} 
allows us to estimate
\begin{align*}
 \underset{t_j\neq t_\ell}{\sideset{}{'}\sum_{0<t_j}\sideset{}{'}\sum_{t_j< t_\ell}} \; \frac{|b_jb_\ell r_\alpha(t_j)r_\alpha(t_\ell)|}{1+|t_j-t_\ell|}&
 \ll\sideset{}{'}\sum_{0<t_j} \frac{|b_jr_\alpha(t_j)|}{t_j^{1/2+\alpha}}\log t_j\\
 &\ll \sum_{0<t_j}  \frac{\log t_j}{t_j^{2+2\alpha}} \; \big(\, |\phi_j(z)|^2+|\phi_j(w)|^2 \,\big)
 \ll_\alpha 1.
\end{align*}

Summarizing we have shown that as $T\to\infty$ 
\begin{equation*}
\frac{1}{T} \int_T^{2T} |f_\alpha(s,\delta)|^2 ds
=
\frac{1}{2}\sideset{}{'}\sum_{0<t_j<e^T} |b_jr_\alpha(t_j)|^2 +
o_\alpha(1).
\end{equation*}

Taking the limit as $T\to\infty$ and using the definition of $r_\alpha(t_j)$ and $b_j$
proves the theorem. Observe that the series on the right is convergent as $T\to\infty$,
again by~\eqref{variance:short-average-cusp-forms}.


\subsection{Variance, cofinite groups}\label{section:variance-cofinite}
We now explain the changes needed in the cofinite case of Theorem
\ref{intro:variance:theorem}: The proof given for cocompact groups extends to cofinite groups
for the analysis of the discrete spectrum. It is in the control of the continuous spectrum that we need the
assumption \eqref{Eisenstein-assumption}. 

We have, from Corollary~\ref{smoothing:corollary}, that
\begin{equation}\label{variance:cofinite-ineq1}
|e_\alpha(s)-f_\alpha(s)|\ll \max_\pm\{\,|g_\alpha^\pm(s,\delta)+P^\pm(s,\delta)+Q_\alpha^\pm(s,\delta)|\,\},
\end{equation}
and we have shown in the proof  of
Theorem~\ref{intro:variance:theorem} in the cocompact case and Lemma~\ref{variance:lemma} that for $\delta=e^{-T}$
\[
\lim_{T\to\infty}\frac{1}{T}\int_T^{2T} |g_\alpha^\pm(s,\delta)|^2ds
=
\lim_{T\to\infty}\frac{1}{T}\int_T^{2T} |P_\alpha^\pm(s,\delta)|^2ds
=0
\]
and
\[
\lim_{T\to\infty}\frac{1}{T}\int_T^{2T} |f_\alpha(s,\delta)|^2 ds
=
2\pi  \sideset{}{'}\sum_{0<t_j} \frac{|\Gamma(it_j)|^2}{|t_j^\alpha\Gamma(3/2+it_j)|^2} \; \bigg|\sum_{t_{j'}=t_j}\phi_{j'}(z)\overline{\phi_{j'}(w)}\;\bigg|^2.
\]
We will show that also the contribution coming from the Eisenstein
series is negligible, namely that 
for $\delta=e^{-T}$ we have
\[
\lim_{T\to\infty}\frac{1}{T}\int_T^{2T} |Q_\alpha^\pm(s,\delta)|^2 ds = 0.
\]
This, using~\eqref{variance:cofinite-ineq1} and Cauchy-Schwartz
inequality, will give the result.

To this end we will show that
\begin{equation}\label{variance:cofinite-negligible-eisenstein}
\lim_{T\to\infty}\frac{1}{T}\int_T^{2T} \Big|\int_\RR h_\alpha^{\prime\pm}(t)E_\mathfrak{a}(t)dt\Big|^2 ds=0.
\end{equation}
We will also use the crude bound $|\tilde{h}_\delta(t)|\ll 1$
for every $0<\delta<1$ and $t\in\RR$.
For $T>2$ we find, using Lemma \ref{shc:products:lemma3}, that

\begin{align*}
\frac{1}{T}\int_T^{2T} &\Big|\int_\RR h_\alpha^{\prime \pm}(t)E_\mathfrak{a}(t)dt\Big|^2 ds
=
\int_\RR \tilde{h}_\delta(t_1)E_\mathfrak{a}(t_1)
\int_\RR \tilde{h}_\delta(t_2)E_\mathfrak{a}(t_2)
\\
& \hspace{5cm}\times
\frac{1}{T}\int_T^{2T} h'_{\alpha,s\pm\delta}(t_1)\overline{h'_{\alpha,s\pm\delta}(t_2)}\, ds \, dt_1 \, dt_2
\\
&\ll
\int_\RR \frac{|E_\mathfrak{a}(t_1)|}{|t_1|^{1+\alpha}(1+\sqrt{|t_1|})}
\int_\RR \frac{|E_\mathfrak{a}(t_2)|}{|t_2|^{1+\alpha}(1+\sqrt{|t_2|})}
\\
&\quad\quad\quad\quad \times
\left(
\frac{1}{1+T|t_1-t_2|}
+
\frac{1}{1+T|t_1+t_2|}
+
\frac{1}{1+\Gamma(\alpha)T^{2-2\alpha}|t_1t_2|^{1-\alpha}}
\right) \, dt_1 \, dt_2.
\end{align*}
Since $|E_\mathfrak{a}(t)|=|E_\mathfrak{a}(-t)|$ we can bound by the slightly simpler expression
\begin{equation}\label{variance:cofinite-eq2}
\begin{gathered}
\frac{1}{T}\int_T^{2T} \Big|\int_\RR h_\alpha^{\prime \pm}(t)E_\mathfrak{a}(t)dt\Big|^2 ds
\ll
\int_0^\infty \frac{|E_\mathfrak{a}(t_1)|}{t_1^{1+\alpha}(1+\sqrt{t_1})}
\int_0^\infty \frac{|E_\mathfrak{a}(t_2)|}{t_2^{1+\alpha}(1+\sqrt{t_2})}
\\
\times
\left(
\frac{1}{1+T|t_1-t_2|}
+
\frac{1}{1+\Gamma(\alpha)T^{2-2\alpha}|t_1t_2|^{1-\alpha}}
\right) \, dt_1 \, dt_2.
\end{gathered}
\end{equation}
For $x>0$ we have $(1+x)^{-1}\leq x^{-r}$ for all
$0\leq r \leq 1$,
so choosing $r=1/2$ we find -- using that $|E_\mathfrak{a}(t)|=O(|t|)$
for $|t|<1$ and the local Weyl law -- that
\[
\begin{gathered}
\int_0^\infty \int_0^\infty
\frac{|E_\mathfrak{a}(t_1)E_\mathfrak{a}(t_2)|}{\big(t_1^{1+\alpha}(1+\sqrt{t_1})\big)\big(t_2^{1+\alpha}(1+\sqrt{t_2})\big)\big(1+\Gamma(\alpha)T^{2-2\alpha}|t_1t_2|^{1-\alpha}\big)} \, dt_1 \, dt_2
\\
\ll
\frac{1}{\Gamma(\alpha)^r T^{r(2-2\alpha)}}
\left(
\int_0^\infty \frac{|E_\mathfrak{a}(t)|}{t^{1+\alpha+r(1-\alpha)}(1+\sqrt{t})}dt
\right)^2
\\
\ll
\frac{1}{\Gamma(\alpha)^{1/2} T^{1-\alpha}}
\left(
\int_0^1 \frac{1}{t^{(1+\alpha)/2}} dt
+
\int_1^\infty \frac{|E_\mathfrak{a}(t)|}{t^{2+\alpha/2}} dt
\right)^2
\ll
\frac{1}{\Gamma(\alpha)^{1/2} T^{1-\alpha}}.
\end{gathered}
\]
In order to estimate the remaining part of~\eqref{variance:cofinite-eq2} we will use the Hardy-Littlewood-P\'olya
inequality (\cite[Theorem 382.]{hardy_inequalities_1952}). This implies that
given $0<\sigma<1$ and $p=2/(2-\sigma)$, every non-negative function $f$ satisfies
\begin{equation}\label{variance:cofinite-HLP-ineq}
\int_0^\infty\int_0^\infty \frac{f(x)f(y)}{|x-y|^\sigma}dxdy
\ll_\sigma
\left(\int_0^\infty f(x)^pdx\right)^{2/p}.
\end{equation}
Applying first $(1+x)^{-1}\leq x^{-\sigma}$, and then
\eqref{variance:cofinite-HLP-ineq} we find
\eqref{Eisenstein-assumption}  
\begin{equation}
\begin{gathered}
\int_0^\infty \int_0^\infty
\frac{|E_\mathfrak{a}(t_1)E_\mathfrak{a}(t_2)|}{(t_1t_2)^{1+\alpha}(1+\sqrt{t_1})(1+\sqrt{t_2})\big(1+T|t_1-t_2|\big)}\, dt_1 \, dt_2
\\
\ll
\frac{1}{T^\sigma}\left(\int_0^\infty \frac{|E_\mathfrak{a}(t)|^p}{t^{(1+\alpha)p}(1+\sqrt{t})^p}dt\right)^{2/p}
\end{gathered}
\end{equation}
If we choose $p$ as in \eqref{Eisenstein-assumption} (and
correspondingly $\sigma=2-2/p$) the last integral is finite since we can bound
\[
\int_0^\infty \frac{|E_\mathfrak{a}(t)|^p}{t^{(1+\alpha)p}(1+\sqrt{t})^p}dt
\ll
\int_0^1 \frac{dt}{t^{\alpha p}} + \int_1^\infty
\frac{|E_\mathfrak{a}(t)|^p}{t^{(1+\alpha)p}(1+\sqrt{t})^p}dt
\ll_\alpha 1
\]
where for the first term we have used $E_\mathfrak{a}(t)=O(|t|)$ for
$|t|\leq 1$, and $p<\alpha^{-1}$, 
and in the second term we have used the bound in
assumption \eqref{Eisenstein-assumption}.
Summarizing, we have proven that
\[
\frac{1}{T}\int_T^{2T} \Big|\int_\RR h_\alpha^{\prime\pm}(t)E_\mathfrak{a}(t)dt\Big|^2 ds
\ll
\frac{1}{T^{1-\alpha}}+\frac{1}{T^{2-2/p}}.
\]
Taking the limit as $T\to\infty$ we obtain~\eqref{variance:cofinite-negligible-eisenstein},
and this concludes the proof of  theorem \ref{intro:variance:theorem}.


\section{Hybrid limits}\label{section:hybrid-limits}

In the previous sections we have shown that for every $0<\alpha<1$
the variance of $e_\alpha(s)$ exists and is finite.
Take for simplicity $z=w$.
We would like to investigate the limit as $\alpha\to 0$ of $\mathrm{Var}(e_\alpha(s,z,z))$,
and conclude that the variance of $e(s,z,z)$ should be given by
\begin{equation}\label{hybrid:sum1}
\mathrm{Var}(e(s,z,z))=
\sideset{}{'}\sum_{0<t_j}
\frac{|\Gamma(it)|^2}{|\Gamma(3/2+it)|^2}
\left(\sum_{t_{j'}=t_j}|\phi_{j'}(z)|^2\right)^2,
\end{equation}
This involves an interchanging of limits that we do not know how to justify,
and so we content ourselves with studying the sum appearing on the right hand side,
and with giving a partial result in direction of \eqref{hybrid:sum1} in
Proposition \ref{hybrid:prop} below.
We cannot even prove that the sum is finite, unless we make assumptions on the eigenfunctions $\phi_j$. It turns out that
the sum barely fails to be convergent:
if we assume
\begin{equation}\label{hybrid:eq1}
\sum_{t_j<T}\!\!\raisebox{5pt}{\scalebox{0.8}{$\prime$}}\;\;
\bigg(\sum_{t_{j'}=t_j}|\phi_{j'}(z)|^2\bigg)^2\;\ll T^{3-\delta}
\end{equation}
for some positive $\delta>0$, then \eqref{hybrid:sum1} becomes finite.

Observe that condition \eqref{hybrid:eq1} with $\delta=0$ is true, in view of 
\eqref{b_j-bound} and \eqref{local-Weyl-law}.

For groups like $\Gamma=\mathrm{PSL}(2,\ZZ)$ 
it is expected that we have strong
bounds on the sup-norm and the multiplicity of eigenfunctions: It is
expected that for any
$0<\delta_1,\delta_2<1/2$ we have 
\begin{equation}
  \label{sub-norm-bound}
  \lvert \phi_j(z)\rvert \ll_z t_j^{1/2-\delta_1} 
\end{equation}
and 
\begin{equation}
  \label{multiplicity-bound}
  m({t_j})= \sum_{t_{j'}=t_j} 1\ll_z t_j^{1/2-\delta_2} 
\end{equation}

Iwaniec and Sarnak \cite{iwaniec_l_1995} has proved \eqref{sub-norm-bound} with
$\delta_1=1/12$, but we know no non-trivial bounds towards
\eqref{multiplicity-bound}. If we knew \eqref{sub-norm-bound} and
\eqref{multiplicity-bound} with $2\delta_1+\delta_2>1/2$ the convergence
of \eqref{hybrid:eq1} would follow.

We conclude this section with the following proposition,
which we state only for cocompact groups:
\begin{prop}\label{hybrid:prop}
Let $\Gamma$ be a cocompact Fuchsian group and let $z\in\HH$.
Assume that \eqref{hybrid:eq1} holds for $\Gamma$.
Let $\alpha=\alpha(T)$ such that
\begin{equation}\label{hybrid:condition}
\lim_{T\to\infty}\alpha(T)=0,
\quad
\frac{1}{\alpha(T)e^{2T\alpha(T)}} \ll 1.
\end{equation}
Then we have
\[\limsup_{T\to\infty}\frac{1}{T}\int_T^{2T}|e_{\alpha(T)}(s)|^2ds<\infty.\]
\end{prop}

\begin{proof}
In \eqref{variance:double-sum} 
we can control the dependence on $\alpha$ in the last sum.
We have indeed
\begin{align*}
\underset{t_j\neq t_\ell}{\sideset{}{'}\sum_{0<t_j<e^T}\sideset{}{'}\sum_{0<t_\ell<e^T}} \;
\frac{|b_jb_\ell r_\alpha(t_j)r_\alpha(t_\ell)|}{1+T|t_j-t_\ell|}
&\ll
\!\sideset{}{'}\sum_{0<t_j<e^T}\sideset{}{'}\sum_{t_j<t_\ell<e^T} \;
\frac{|b_jb_\ell r_\alpha(t_j)r_\alpha(t_\ell)|}{1+T|t_j-t_\ell|}\\
&\ll
\!\sideset{}{'}\sum_{0<t_j<e^T}\frac{\lvert b_j\rvert}{t_j^{3/2+\alpha}}
\frac{1}{t_j^{1/2+\alpha}}\left(1+\frac{\log(t_j+1)}{T}\right)\\
&\ll 1+ \frac{1}{\alpha e^{2\alpha T}}
\end{align*}
where we have used Lemma \ref{cramer-lemma} and \eqref{local-Weyl-law}.
The implied constant is now independent of $\alpha$.
Using this, and adding to  \eqref{variance:double-sum}
the estimates from \eqref{variance:cocompact-P-square-bounds},
\eqref{variance:lemma-eq1}, and \eqref{variance:lemma-eq2}, we obtain,
for $\delta=e^{-T}$,
\[
\begin{gathered}
\frac{1}{T}\int_T^{2T} |e_\alpha(s)|^2\,ds
\ll
1 + \frac{1}{T} +\frac{1}{\alpha e^{2\alpha T}}+ \frac{1}{T^{2-2\alpha}}
+ \frac{T^{2+2\alpha}}{e^T} +\frac{T^{2\alpha}}{e^{\eps_\Gamma T}}.
\end{gathered}
\]
Take now $\alpha=\alpha(T)$ as in the statement.
Condition \eqref{hybrid:condition} is sufficient for all the terms,
in particular the third one, to be bounded as $T\to\infty$,
and so we conclude
\[\limsup_{T\to\infty}\frac{1}{T}\int_T^{2T}|e_{\alpha(T)}(s)|^2ds<\infty\]
which is the claim.
\end{proof}


\section{Limiting Distribution}\label{section:limiting-distribution}

We are now ready to prove Theorem \ref{intro:limiting-distribution}:
In proving Theorem~\ref{intro:variance:theorem} 
(section \ref{section:variance}) 
 we have shown that if we write
\[
e_\alpha(s) = \sideset{}{'}\sum_{0<t_j<X} 
\Re(r_\alpha(t_j)e^{it_js})b_j + \mathcal{E}(s,X)
\]
for $r_\alpha(t_j),b_j\in\CC$, defined as in~\eqref{variance:r_j-definition}
and~\eqref{b_j-bound}, then we have
\begin{equation}\label{so-far}\lim_{T\to\infty}\frac{1}{T}\int_T^{2T}|\mathcal{E}(s,e^T)|^2ds=0.\end{equation}
We claim that also 
\begin{equation}\label{integral-zero}\lim_{T\to\infty}\frac{1}{T}\int_0^{T}|\mathcal{E}(s,e^T)|^2ds=0.\end{equation}

To see this we note that
\begin{equation*}
  \frac{1}{T}\int_0^T|\mathcal{E}(s,e^T)|^2ds= \sum_{n=1}^\infty\frac{1}{2^n}\frac{2^n}{T}\int_{\frac{T}{2^n}}^{\frac{2T}{2^n}}|\mathcal{E}(s,e^T)|^2ds
\end{equation*}
We claim that for $T'\leq T$ we have
\begin{equation}\label{intermediate}
  \frac{1}{T'}\int_{T'}^{2T'}|\mathcal{E}(s,e^T)|^2ds\to 0 \textrm{ as
  }T'\to \infty
\end{equation} where the convergence is uniform in $T\geq T'$. By the
dominated convergence theorem we may then conclude \eqref{integral-zero}. To
see \eqref{intermediate} we note that 
\begin{align*}
  \frac{1}{T'}\int_{T'}^{2T'}|\mathcal{E}(s,e^T)|^2ds\leq
  \frac{2}{T'}&\int_{T'}^{2T'}|\mathcal{E}(s,e^{T'})|^2ds\\ &+\frac{2}{T'}\int_{T'}^{2T'}|\sideset{}{'}\sum_{e^{T'}<t_j\leq
  e^T}\Re(r_\alpha(t_j)e^{it_js})b_j|^2ds.
\end{align*}
The first term does not depend on $T$ and tends to 0 as
$T'\to\infty$ by \eqref{so-far}. The second term can be analyzed as in Section
\ref{subsection-variance-cocompact} and we find that this term goes to
zero uniformly in $T$. This proves \eqref{intermediate} and proves
therefore \eqref{integral-zero}.

Eq.  \eqref{integral-zero} implies that $e_\alpha(s)$ is in the closure of the set
\[\left\{\sum_{\textrm{finite}} r_ne^{is\lambda_n}:\lambda_n\in\RR, r_n\in\CC\right\}\]
with respect to the seminorm
\begin{equation*}\label{distribution:seminorm}
\|f\|=\limsup_{T\to\infty}\left(\frac{1}{T}\int_0^{T}|f(s)|^2ds\right)^{1/2}.
\end{equation*}
In other words, $e_\alpha(s)$ is an almost periodic function
with respect to \eqref{distribution:seminorm}, i.e. a $B^2$-almost periodic function.
We can then apply \cite[Theorem 2.9]{akbary_limiting_2014}
and  conclude that $e_\alpha(s)$ admits a limiting distribution.
The last part of the theorem is a direct consequence of the fact that
$e_\alpha(s)$ is bounded for $1/2<\alpha<1$ (see Theorem \ref{intro:pointwise:theorem}).

\end{document}